\documentclass[11pt,letterpaper,reqno]{article}

\usepackage[letterpaper, left = 1.25in, right = 1.25in, top = 1.25 in, bottom = 1.25in]{geometry}
\usepackage[affil-it]{authblk}
\usepackage[sc]{mathpazo}
\usepackage[T1]{fontenc}
\usepackage{enumerate}

\usepackage{amsmath, amsthm, amssymb}
\usepackage[color = white]{todonotes}
\usepackage[sort&compress,numbers]{natbib}
\usepackage[linktocpage=true]{hyperref}
\hypersetup{
 colorlinks,
 linkcolor=blue,          
 citecolor=red!50!black,       
 filecolor=black,   
 urlcolor=black, 
 pdftitle={},
 pdfauthor={},
 pdfcreator={},
 pdfsubject={},
 pdfkeywords={}
}

\def\sss{\scriptscriptstyle}

\newcommand{\inner}[1]{\langle #1 \rangle}
\newcommand{\binner}[1]{\big\langle #1 \big\rangle}

\newcommand{\expt}[1]{\ensuremath{\mathbb{E}\left[#1\right]}}

\newcommand{\var}{\mathrm{Var}}

\newcommand{\PR}{\ensuremath{\mathbb{P}}}
\newcommand{\E}{\ensuremath{\mathbb{E}}}
\newcommand{\R}{\ensuremath{\mathbb{R}}}
\newcommand{\N}{\ensuremath{\mathbb{N}}}

\newcommand{\1}{\ensuremath{\bld{1}}}

\newcommand{\dto}{\ensuremath{\xrightarrow{d}}}

\newcommand{\blue}[1]{{\color{black}#1}}
\newcommand{\dif}{\mathrm{d}}
\newcommand{\bld}[1]{\boldsymbol{#1}}

\newcommand{\bx}{\bld{x}}

\newcommand{\bzero}{\bld{0}}

\DeclareMathOperator*{\argmax}{arg\,max}

\newcommand{\newsigma}{\gamma}

\usepackage{stackengine} 
\newcommand{\oast}{\star}

\newcommand{\lag}{\kappa}
\newcommand{\minn}{m_n}
\newcommand{\Maxn}{M_n}
\newcommand{\ve}{\varepsilon}
\newcommand{\aij}{a_{ij}^n}

\newcommand{\aik}{a_{ik}^n}

\newcommand{\akj}{a_{kj}^n}
\newcommand{\aii}{a_{ii}^n}

\newcommand{\degreei}{d_n(i)}
\newcommand{\degreeone}{d_n(1)}

\newcommand{\sgn}{\mathrm{sgn}}
\newcommand{\one}{\mathbf{1}}

\newcommand{\bA}{\bar{A}}
\newcommand{\bd}{\bar{\boldsymbol{d}}}
\newcommand{\bm}{\bar{m}}
\newcommand{\ba}{\bar{a}}

\newcommand{\hh}{\boldsymbol{h}}
\newcommand{\uu}{\boldsymbol{u}}

\newcommand{\vv}{\boldsymbol{v}}
\newcommand{\xx}{\boldsymbol{x}}
\newcommand{\yy}{\boldsymbol{y}}
\newcommand{\zz}{\boldsymbol{z}}
\newcommand{\ww}{\boldsymbol{w}}

\newcommand{\oP}{o_{\sss \PR}}
\newcommand{\OP}{O_{\sss \PR}}

\newcommand{\opern}{\|A_n\|_{r\to p}}

\newcommand{\opera}[1]{\|#1\|_{r\to p}}

\newcommand\blfootnote[1]{%
  \begingroup
  \renewcommand\thefootnote{}\footnote{#1}%
  \addtocounter{footnote}{-1}%
  \endgroup
}

\usepackage{environ}
\NewEnviron{eq}{%
\begin{equation}\begin{split}
  \BODY
\end{split}\end{equation}
}
\newtheorem{theorem}{Theorem}
\newtheorem{lemma}[theorem]{Lemma}
\newtheorem{proposition}[theorem]{Proposition}
\newtheorem{corollary}[theorem]{Corollary}
\newtheorem{assumption}[theorem]{Assumption}
\newtheorem{remark}[theorem]{Remark}
\newtheorem{fact}[theorem]{Fact}
\newtheorem*{claim*}{Claim}
\newtheorem{claim}[theorem]{Claim}
\newtheorem{defn}[theorem]{Definition}

\let\plainqed\qedsymbol
\newcommand{\claimqed}{$\lrcorner$}
\newenvironment{claimproof}{\begin{proof}\renewcommand{\qedsymbol}{\claimqed}}{\end{proof}\renewcommand{\qedsymbol}{\plainqed}}

\numberwithin{equation}{section}
\numberwithin{theorem}{section}

\begin{document}

\title{On $r$-to-$p$ norms of random matrices with nonnegative entries: Asymptotic normality and $\ell_\infty$-bounds for the maximizer}
\author{Souvik Dhara$^{1}$, Debankur Mukherjee$^2$, Kavita Ramanan$^3$}

\date{}

\maketitle
\begin{abstract}
For an $n\times n$ matrix $A_n$, the $r\to p$ operator norm is defined as 
$$\opera{A_n}:= \sup_{\bx \in \R^n:\|\bx\|_r\leq 1 } \|A_n\xx\|_p\quad\mbox{for}\quad r,p\geq 1.$$
For different choices of $r$ and $p$, this norm corresponds to key quantities that arise in diverse applications including matrix condition number estimation, clustering of data, and construction of oblivious routing schemes in transportation networks. This article considers $r\to p$ norms of symmetric random matrices with nonnegative entries, including  adjacency matrices of  Erd\H{o}s-R\'enyi random graphs, matrices with positive sub-Gaussian entries, and certain sparse matrices. For $1< p\leq r< \infty$, the asymptotic normality,  as $n\to\infty$, of the appropriately centered and scaled norm $\opera{A_n}$ is established.  When $p \geq 2$, this is  shown to imply, as a corollary,  asymptotic normality of the solution to the $\ell_p$ quadratic maximization problem, also known as the $\ell_p$ Grothendieck problem.  Furthermore, a sharp $\ell_\infty$-approximation bound for the unique maximizing vector in the definition of 
$\opera{A_n}$ is obtained, and may be viewed as an $\ell_\infty$-stability result of the maximizer under random perturbations of the matrix with mean entries. This result, which may be of independent interest, is in fact shown to hold for a broad class of deterministic sequences of matrices having certain asymptotic expansion properties. The results obtained can be viewed as a  generalization of the seminal results of F\"{u}redi and  Koml\'{o}s (1981) on asymptotic normality of the largest singular value of a class of symmetric random matrices, which corresponds to the special case $r=p=2$ considered here. In the general case with $1< p\leq r < \infty$,  spectral methods are no longer applicable, and so a new approach is developed involving a refined convergence analysis of a nonlinear power method and  a perturbation bound on the maximizing vector, which may be of independent interest. 
\blfootnote{$^1$Purdue University, \emph{Email:} \href{mailto:sdhara@purdue.edu}{sdhara@purdue.edu}}
\blfootnote{$^2$Georgia Institute of Technology, \emph{Email:}  \href{mailto:debankur.mukherjee@isye.gatech.edu}{debankur.mukherjee@isye.gatech.edu}}
\blfootnote{$^3$Brown University, \emph{Email:} \href{mailto:kavita_ramanan@brown.edu}{kavita\_ramanan@brown.edu}}
\blfootnote{\emph{Acknowledgements.} Mukherjee was partially supported by the NSF grants CIF-2113027 and CPS-2240982.
Dhara was partially supported by Vannevar Bush Faculty Fellowship ONR-N00014-20-1-2826, Simons-Berkeley Research Fellowship and Vannevar Bush Faculty Fellowship ONR-N0014-21-1-2887. 
Ramanan was partially supported by the NSF via grant  DMS-1954351.}
     \end{abstract}

\newpage
 \setcounter{tocdepth}{1}
 \tableofcontents

\section{Introduction}
\subsection{Problem statement and motivation}
For any $n\times n$ square matrix $A_n$ and $r, p \geq 1$, the $r\to p$ operator norm of $A_n$ is defined as   
\begin{eq}\label{eq:op-norm-def}
\opera{A_n}:= \sup_{\|\bx\|_r\leq 1 } \|A_n\xx\|_p. 
\end{eq}
For different values of $r$ and $p$, 
the $r\to p$ operator norm represents key quantities that 
arise in a broad range of disciplines.  
For example, when $p=r=2$,  this corresponds to the largest singular value of the matrix $A_n$, which has been studied extensively for decades. 
On the other hand, when $p$ is the H\"older conjugate of $r$, that is, $p = r/(r-1)$, and $A_n$ has nonnegative entries and $A_n^TA_n$ is irreducible, then we will see (in Proposition \ref{prop:groth} and Section \ref{sec:Groth}) that 
this problem reduces to the famous $\ell_r$ Grothendieck problem~\cite[Section 5]{KN12}, 
which has inspired a vibrant line of research in the optimization community. 
Two special cases of the $\ell_r$ Grothendieck problem, namely when $r = 2$ and $r=\infty$, relate to  spectral partitioning~\cite{DH73, Fiedler73} and correlation clustering~\cite{CW04}, respectively, 
and the case of general $r \in (2,\infty)$ can be viewed as a smooth interpolation
between these two  clustering criteria.  
Further, this problem is also related to finding ground states in statistical physics problems. 
Another interesting special case is when $p=r$, which has been a classical topic; 
see \cite{Wilf70,SS62} for general inequalities involving the $p\to p$ norm, 
 \cite{Hig87} for applications of these norms to matrix condition number estimation, which is crucial for computing  perturbations of solutions to linear equations, 
 and  \cite{Hig92,Boyd74} for algorithms to approximate such norms.
Other prime application areas are: construction of oblivious routing schemes in  transportation networks for the $\ell_p$ norm~\cite{ER09, BV11, Racke08, GHR06}, and data dimension reduction or sketching of these norms, with applications to 
the streaming model and robust regression~\cite{KMW18}. 
Understanding the computational complexity of calculating $r\to p$ norms has generated  immense recent  interest in theoretical computer science.  
We refer the reader to~\cite{KN12} for a detailed account of the applications, approximability results, and Grothendieck-type inequalities for this norm.
In general, this problem is NP-hard;  even providing a constant-factor approximation algorithm for this problem is hard \cite{BV11,HO10,BGGLT19}.  
However, for the case considered in this article, namely 
matrices with nonnegative entries and $1< p\leq r<\infty$,  this problem can be solved in polynomial time  \cite{BV11,Boyd74}. 
The cases when $p=1$ and $r\geq 1$ are equivalent to the cases $p\leq \infty$ and $r=\infty$~\cite[Lemma 8]{KMW18}.
 These cases are trivial for nonnegative matrices and hence, we do not consider them in this article.

 The analysis of this norm for random matrices is motivated from a statistical point of view. Indeed, asymptotic results  on spectral statistics and eignevectors form the bedrock of methods in high-dimensional statistics (see~\cite{Vershynin18, Wainwright19, BG11} for a sample of
 the vast literature in this area).   
Further, it is worth mentioning the seminal work of F\"{u}redi and Koml\'{o}s~\cite{FK81}, where asymptotic normality of the largest eigenvalue was first established for matrices with i.i.d.~entries.
Subsequently, this result has been extended to adjacency matrices of sparse Erd\H{o}s-R\'enyi random graphs~\cite{EKYY13}, stochastic block model~\cite{Tang18}, and rank-1 inhomogeneous random graphs~\cite{CCH19}.
In the context of general $r\to p$ norms for random matrices, the $p>r$ case has received much attention. 
For matrices with bounded mean-zero independent entries, asymptotic bounds on the $2\to p$ norm was  established in~\cite{BGN75} for $2\leq p<\infty$. 
For $1<r\leq 2\leq p<\infty$ and matrices having  i.i.d.~entries, $\|A_n\|_{r\to p}$ is known to concentrate around its median \cite{Mec04}. 
  Furthermore, in this regime, refined bounds on the expected $r\to p$ norm of centered Gaussian random matrices have been obtained in~\cite{GHLP17} and later extended to log-concave random matrices with dependent entries in~\cite{Strzelecka19}. 
\\

Another quantity of considerable interest is the maximizing vector in~\eqref{eq:op-norm-def}. 
For example, in the $p=r=2$ case, eigenvectors of  adjacency matrices of graphs are known to play a pivotal role in developing efficient graph algorithms, such as  spectral clustering~\cite{SM00, Luxburg07}, spectral partitioning~\cite{DH73, Fiedler73, McSherry01, PSL90}, PageRank~\cite{PBMW99}, and community detection~\cite{New06, Newman06b}. 
 Eigenvectors of random matrices can be viewed as perturbations of eigenvectors  of the expectation matrix,
 in the presence of additive random noise in the entries of the latter.  
  The study of eigenvector perturbation bounds can be traced back to 
the classical  Rayleigh-Schr\"odinger  theory \cite{Rayleigh96,Schrodinger26} in quantum mechanics, which  gives asymptotic perturbation bounds in  the $\ell_2$-norm, as the signal to noise ratio increases.
Non-asymptotic perturbation bounds in the $\ell_2$-norm were derived later in a landmark result~\cite{DK70}, popularly known as the Davis-Kahan $\sin\Theta$ theorem.
When the perturbation is random, the above deterministic results typically yield suboptimal bounds.
 Random perturbations of low-rank matrices has recently been analyzed in~\cite{OVW18}.
However, norms that are not unitary-invariant, such as the $\ell_\infty$-norm, as considered in this paper, are typically outside the scope of the above works,  although they  are of significant interest in statistics and machine learning.
The $\ell_\infty$-norm bounds in the case of low-rank matrices have been studied recently in~\cite{FWZ17,CTP19,EBW18, AFWZ19, Zhong17, Mit09}, and 
~\cite{OVW16, FWZ17,AFWZ19} contain extensive discussions on such perturbation bounds on eigenvectors (or singular vectors) and their numerous applications in statistics and machine learning.

\subsection{Our contributions}
Fix $1 < p \leq r < \infty$.  We now elaborate on the  two main results of the current article, namely  
asymptotic normality of a suitably scaled and centered version of $\opera{A_n}$,  and approximation of the corresponding maximizing 
vector.

\paragraph*{{\normalfont(1)} Asymptotic normality.} 
Given a sequence of symmetric nonnegative  random matrices $(A_n)_{n \in \N}$,  our first set of results establishes asymptotic normality of the  scaled norm $\opera{\bar{A}_n} := n^{-(\frac{1}{p}-\frac{1}{r})} \opera{A_n}$ when $1 < p \leq r < \infty.$  
Specifically, let $A_n$  have zero diagonal entries and independent and identically distributed (i.i.d.) off-diagonal entries \blue{subject to the symmetry constraint} that have mean $\mu_n$, variance $\sigma_n^2 > 0$. Under certain moment bounds on the distribution of the matrix entries, and a control
 on the asymptotic sparsity of the  matrix sequence, expressed in terms of 
conditions on the (relative) rates at which $\sigma_n^2$ and $\mu_n$ can decay to zero,
it is shown in Theorem \ref{thm:asymptotic-normality} that 
as $n \rightarrow \infty$, 
\begin{equation}
    \label{eq-mainresult}
   \frac{1}{\sigma_n} \left(\opera{\bar{A}_n} - \alpha_n(p,r) \right) \dto Z \sim \mathrm{Normal} \big(0, 2\big), 
\end{equation}
where $\dto$ denotes convergence in distribution,  and 
\begin{equation}
    \label{def-alphan}
 \alpha_n (p,r) := (n-1) \mu_n + \frac{1}{2}\left( p-1 + \frac{1}{r-1} \right) \frac{\sigma_n^2}{\mu_n}. 
\end{equation} 
An extension of the above result for random matrices with inhomogeneous variance profile is also provided in 
  Theorem \ref{thm:asymptotic-normality-inhom}. 
  In this case, however, the  matrix is required to be dense. 
  
  A result of this flavor appears to have first been established 
  in the seminal work of  F\"{u}redi and Koml\'{o}s \cite{FK81} 
  for the special case $r = p = 2$, where $\|\bar{A}_n\|_{2 \to 2} = \|A_n\|_{2 \to 2}$ represents  $\lambda_1^{(n)}$, the largest eigenvalue of $A_n$.   Using spectral methods, it is shown in 
  \cite[Theorem 1]{FK81} that under the assumption that 
  $A_n$ is a symmetric $n \times n$ random matrix with zero diagonal entries, independent, uniformly bounded off-diagonal entries having a common positive mean $\mu>0$ and variance $\sigma^2>0$ (with $\mu, \sigma$ not depending on $n$), the limit \eqref{eq-mainresult} holds with $r=p=2$, $\sigma_n = \sigma$, and $\alpha_n(2,2) = (n-1)\mu + \sigma^2/\mu,$  which coincides with the definition in  
  \eqref{def-alphan}, when one sets  $\mu_n = \mu$  and $\sigma_n^2 = \sigma^2$.  
      Even for the case $p=r=2$, our result extends the asymptotic normality 
 result of F\"{u}redi and Koml\'{o}s~\cite{FK81}
in three directions: it allows for (a) sequences of possibly sparse matrices $(A_n)_{n\in\N}$, that is with $\mu_n \rightarrow 0$; (b) independent and identically distributed (i.i.d.)~off-diagonal entries satisfying suitable moment conditions, but with possibly unbounded support; (c) independent entries with possibly different variances, having a dense variance profile.  
Throughout, the assumption that the diagonal entries are identically zero is only made for simplicity of notation;   the result of \cite{FK81} also allows for the diagonal entries to be drawn from another independent sequence of entries with a different common positive mean and uniformly bounded support on the diagonal,  and an analogous extension can  also be accommodated in our setting; see Remark \ref{rem:diag}.   Moreover, we do not necessarily
  identify the optimal level of sparsity, see Remark \ref{rem:assumption-1} for
  an elaboration of this point.
  
It is worth mentioning two interesting aspects of the limit in 
\eqref{eq-mainresult}.    
Consider  the setting where $\mu_n = \mu > 0$ and $\sigma_n^2 = \sigma^2 > 0$, as considered in \cite{FK81}. 
First, note that 
while 
$\opera{\E[\bar{A}_n]}=(n-1) \mu$, 
and $\opera{\bar{A}_n}/\opera{\E[\bar{A}_n]}$ converges in probability to $1$, 
the  centering $\alpha_n(p,r)$ is strictly 
larger than $(n-1) \mu$ by a $\Theta(1)$ asymptotically non-vanishing amount.   
Second, whereas the centering $\alpha_n(p,r)$ for $\opera{\bar{A}_n}$ 
is $\Theta (n)$,  the Gaussian fluctuations of  $\opera{\bar{A}_n}$ are
only $\Theta(1)$, having variance $2$.  
 Both these properties also hold for  the 
case $r = p = 2$ analyzed in \cite{FK81}, and the 
second property  can be seen as 
 a manifestation of the rigidity phenomenon for  eigenvalues of random matrices. 
 This has subsequently been shown to occur in a variety 
 of other random matrix models, but there is {\em a priori} no reason to expect this to generalize to the non-spectral setting of 
 a general $r \to p$ norm.  
While spectral methods can be used in the case $p=r=2$, they are no longer applicable in the general $r \to p$ norm setting. 
Thus, we develop a new approach, which also 
 reveals some  key reasons for these phenomena  to occur, and  brings to light when the shift and rigidity  properties will fail when considering sparse sequences of matrices. (see Remark~\ref{rem:mean-shift}).

\paragraph*{{\normalfont(2)} Approximation of the maximizing vector.}
Our second set of results are summarized in  Theorem~\ref{thm:maximizer-vector}, 
 which  provides an $\ell_\infty$-approximation of the maximizing vector for matrices with i.i.d.~entries, 
 and  Theorem~\ref{thm:maximizer-vector-inhom},  which extends this to random matrices with inhomogeneous variance profiles. 
 These results rely on  Proposition~\ref{prop:vec-close}, which states an approximation result for the maximizer of the $r \to p$ norm, for
 arbitrary (deterministic) sequences of symmetric matrices satisfying certain asymptotic expansion properties.

It is not hard to see that the maximizing vector for the $r\to p$ norm of the expectation matrix  is given by $n^{-1/r}\1$, the scaled $n$-dimensional vector of all 1's. 
Thus, the maximizing vector $\vv_n$ corresponding to the random matrix can be 
viewed as a perturbation of $n^{-1/r}\1$, and  our result can be thought of as an  entrywise  perturbation bound of the maximizing vector for the expectation matrix. 
In contrast with the $p=r=2$ case, 
the unavailability of  spectral methods for the  general  $1<p\leq r<\infty$ case  makes the problem significantly more challenging, which led us to develop a novel approach to characterize the $\ell_\infty$-approximation error
for a sequence of \blue{deterministic} matrices satisfying some general conditions.

\subsection{Notation and organization}
We write $[n]$ to denote the set $\{1,2,\dots,n\}$.
We use the standard notation of $\xrightarrow{\sss\PR}$ and $\xrightarrow{\sss d}$ to denote convergence in probability and in distribution, respectively. 
Also, we often use the Bachmann-Landau notation $O(\cdot)$, $o(\cdot)$, $\Theta(\cdot)$ for asymptotic comparisons.
For two positive deterministic sequences $(f(n))_{n\geq 1}$ and $(g(n))_{n\geq 1}$, we write $f(n)\ll g(n)$ (respectively, $f(n)\gg g(n)$), if $f(n) = o(g(n))$ (respectively, $f(n) = \omega(g(n))$).
For a positive deterministic sequence $(f(n))_{n\geq 1}$, a sequence of random variables 
$(X(n))_{n\geq 1}$ is said to be $\OP(f(n))$ and $\oP(f(n))$, if the sequence $(X(n)/f(n))_{n\geq 1}$ 
is tight and $X(n)/f(n)\xrightarrow{\sss\PR} 0$ as $n\to\infty$, respectively.
\blue{For two sequences of real-valued random variables $(X_n)_{n\geq 1}$ and $(Y_n)_{n\geq 1}$, we will write $X_n\lesssim Y_n$ if there exists some constant $c>0$, such that $\PR(X_n \leq cY_n)\to 1$ as $n\to\infty$.}
$\mathrm{Normal}(\mu, \sigma^2)$ is used to denote normal distribution with mean $\mu$ and variance $\sigma^2$.
For two vectors $\xx = (x_i)_i\in \R^n$ and $\yy=(y_i)_i\in \R^n$, define the `$\oast$' operation as
the entrywise product given by $\zz = \xx \oast \yy  = (x_iy_i)_i\in\R^n$.
Define $\1$ to be the $n$-dimensional vector of all 1's, $J_n := \1\1^T$, and $I_n$ to be the $n$-dimensional identity matrix. 
Also, $1\{\cdot\}$ denotes the indicator function.\\

The rest of the paper is organized as follows. 
In Section~\ref{sec:main} we state the main results and discuss their ramifications.
Section~\ref{sec:outline} provides a high-level outline of the proofs of the main results.
In Section~\ref{sec:prelim} we introduce the basics of the nonlinear power method, which will be
a key tool for our analysis, and present some preliminary results.
Sections~\ref{sec:vector} and~\ref{subs-verification} concern the approximation of the maximizing vector in the deterministic and
random cases, respectively.
Section~\ref{sec:power-method} presents a two-step approximation of the $r\to p$ norm and in particular, identifies a functional of the underlying random matrix that is `close' to the $r\to p$ norm.
In Section~\ref{sec:asym-normality} we prove the asymptotic normality of
this approximating functional.   
Finally, in Section \ref{sec:Groth}, we end by exploring the relation between the $r\to p$ norm and the $\ell_p$ Grothendieck problem.
Some of the involved but conceptually straightforward calculations are deferred  to the appendix.

\section{Main results}\label{sec:main}
 In this section we present our main results.
Section~\ref{ssec:homogeneous} describes results for random matrices with i.i.d.~entries (except possibly the diagonal entries), whereas
 Section~\ref{ssec:inhomogeneous} states extension of the main results  when 
the matrix entries can have inhomogeneity in their variances.
Finally, in Section~\ref{ssec:special} we discuss the implications of our results  in two important special cases.

\subsection{Matrices with i.i.d.~entries}\label{ssec:homogeneous}
We start by stating a general set of assumptions on the sequence of random matrices: 
\begin{assumption} \normalfont \label{assumption-1}
For each $n\geq 1$, let $F_n$ be a distribution supported on $[0,\infty)$ and having finite mean $\mu_n$ and variance $\sigma_n^2$. 
Let $A_n = (\aij)_{i,j = 1}^n$ be a symmetric
random matrix such that
\begin{enumerate}[(i)]

\item \label{assumption1-i} $(\aij)_{i,j= 1,i<j}^n$ are i.i.d.~random variables with common distribution $F_n$. 
Also, $\aii =0$ for all $i\in [n]$.

 \item \label{assumption1-ii}
 $\mu_n =O(1)$, \blue{$\mu_n  = \omega \big(\frac{\log^{2/3} n}{n^{1/3}}\big)$,}
$\sigma_n \geq n^{-\frac{1}{2} + c_0}$ for some constant $c_0>0$, and $\frac{\sigma_n^2}{\mu_n} = O(1)$. 
 

 \item\label{assumption1-iv} There exists $c < \infty$, such that $\E\big[|a_{12}^n-\mu_n|^k\big]\leq \frac{k!}{2} c^{k-2}\sigma_n^2$ for all $k\geq 3$.
\end{enumerate}
\end{assumption} 

\begin{remark}\label{rem:assumption-1}
  \normalfont
Observe that  Assumption~\ref{assumption-1}\eqref{assumption1-ii} is trivially 
    satisfied
    in the dense regime, where $\mu_n = \mu$ and $\sigma_n^2 = \sigma$
    are fixed constants, which was the setting considered by 
    F\"{u}redi and Koml\'{o}s in \cite{FK81}.  The weaker conditions imposed in      Assumption~\ref{assumption-1}\eqref{assumption1-ii} 
    show that  our approach also covers  a broad class of sparse matrices. However, the conditions on the sparsity of the matrices are not necessarily
    optimal, and identifying optimal conditions
    is beyond the scope of this article.
  The  reasons are elaborated below. 
    The lower bound on $\sigma_n$ in Assumption~\ref{assumption-1}~\eqref{assumption1-ii} is required when we apply existing asymptotic results for second largest eigenvalues of random matrices~\cite{LS18} to approximate the operator norm (see the proof of Lemma~\ref{claim:lambda2-bound}), and the condition on $\mu_n$ is required in the proof of Lemma~\ref{lem-verification} (to establish well-connectedness), in the approximation step in Lemma~\ref{lem:errg-deg}, and in the proof of Theorem~\ref{thm:asymptotic-normality}.
\blue{Indeed, this assumption is used in the strongest form in the final step of the proof of Theorem~\ref{thm:asymptotic-normality}; see the two displays below (8.15).}  
The moment conditions in Assumption~\ref{assumption-1}\eqref{assumption1-iv} guarantee concentration of certain relevant polynomials of the matrix elements, which we  use to approximate the operator norm. 
At first sight, they may appear restrictive, but such conditions frequently arise in the literature (cf.~\cite{LS18,AEK19}), for example, when applying Bernstein's inequality.
\end{remark}

\subsubsection{Asymptotic normality of the \texorpdfstring{$r\to p$}{r-to-p} norm}
Our first main result provides a central limit theorem for the $r\to p$ norms of random matrices satisfying Assumption~\ref{assumption-1}. 
Theorem~\ref{thm:asymptotic-normality} is proved in Section~\ref{ssec:asymp-normality}.
\begin{theorem}
\label{thm:asymptotic-normality}
Fix any $1< p\leq r< \infty$.
Consider the sequence of random matrices $(A_n)_{n\in \N}$ satisfying Assumption~\ref{assumption-1} and define 
 $   \bA_n:= n^{-(\frac{1}{p}-\frac{1}{r})} A_n.$
Then, as $n\to\infty$,
 \begin{eq}\label{eq:asymp-normal}
\frac{1}{\sigma_n}\big(\opera{\bA_n}- \alpha_n(p,r)\big) \dto Z \sim \mathrm{Normal}(0,2),
\end{eq} where 
\begin{equation}\label{eq:centering}
    \alpha_n(p,r)= (n-1)\mu_n  + \Big(p-1+\frac{1}{r-1}\Big)\frac{\sigma_n^2}{2\mu_n}.
\end{equation} 
\end{theorem}

\begin{remark}
\label{rem:diag}
\normalfont The assumption that $\aii = 0$ in Theorem~\ref{thm:asymptotic-normality} is not a strict requirement. In fact, 
one can assume $\aii$'s to be independent of $\aij$'s and to be i.i.d.~from some distribution $G_n$  
with nonnegative support, mean $\zeta_n = \Theta(\mu_n^2)$, variance $\rho_n^2 = \Theta(\sigma_n^2)$, and \blue{satisfying the moment condition in Assumption~\ref{assumption-1}~\eqref{assumption1-iv} with $\mu_n$ and $\sigma_n$ replaced by $\zeta_n$ and $\rho_n$, respectively. }
Then \eqref{eq:asymp-normal} holds with 
\begin{eq}
     \alpha_n(p,r)=(n-1)\mu_n  + \zeta_n + \Big(p-1+\frac{1}{r-1}\Big)\frac{\sigma_n^2}{2\mu_n}.
\end{eq}
All our proofs go through verbatim in this case, except for a minor modification to Lemma~\ref{claim:lambda2-bound}, which is addressed in 
Lemma~\ref{claim:lambda2-bound-2}. 
However, assuming the diagonal entries to be 0 saves significant additional notational burden and computational complications. 
For that reason, we will assume $\aii = 0$ throughout the rest of the paper.
\end{remark}

\begin{remark}
\label{rem:mean-shift}
\normalfont
As briefly mentioned in the introduction, 
an intriguing fact to note from Theorem~\ref{thm:asymptotic-normality} is 
that although $\opera{\bar{A}_n}$ is concentrated around $\opera{\E[\bar{A}_n]}$,
on the CLT scale, there is a non-trivial further $O(1)$ shift $\alpha_n(p,r)$ in the mean.
This is consistent with~\cite{FK81} for  the case
$p=r=2$.
As we will see in the proof of Theorem~\ref{thm:asymptotic-normality} in Section~\ref{ssec:asymp-normality}, this additional constant shift arises  from a Hessian term when we perform the Taylor expansion of a suitable approximation of $\opera{A_n}$. 
It is also worth noting that, if $\sigma_n^2\ll \mu_n$ (e.g., when $F_n$ is an exponential distribution with mean $\mu_n\to 0$), this additional shift vanishes,  and thus there may be no shift for certain asymptotically sparse matrix sequences. 
\end{remark}

\begin{remark}
\normalfont
There are two noteworthy phenomena about the asymptotic variance of $\opera{A_n}$. 
First, the asymptotic variance does not depend on $p,r$ beyond the scaling factor $n^{\frac{1}{p} - \frac{1}{r}}$. 
Second, if $p=r$ and we are in the dense setting (i.e., $\mu_n = \mu>0$ and $\sigma_n = \sigma>0$), the asymptotic variance is a $\Theta(1)$ quantity, although the mean is $\Theta(n)$. 
The latter is analogous to the rigidity phenomenon for the largest eigenvalue of random matrices.  
In the $2\to 2$ norm case when the $\aij$ are uniformly bounded, this constant order of the asymptotic variance
can be understood from the application of 
the bounded difference inequality (see \cite[Corollary 2.4, Example 2.5]{vanHandel14}, which considers the case when $\aij$ are Bernoulli).
However, as we see  in \cite[Example 2.5]{vanHandel14}, in order to bound
 the expected change in the operator norm after changing one entry of the matrix, the fact that $\ell_2$ is a Hilbert space is crucial, and this method does not generalize directly for $\ell_p$ spaces with $p\neq 2$. 
 Nevertheless, as we have shown in  Theorem \ref{thm:asymptotic-normality},
 the variance still turns out to be $\Theta(1)$ for the general $p=r$ case in the dense setting. 
\end{remark}

\subsubsection{The maximizing vector}
The second main result is an $\ell_\infty$-approximation of the maximizing  vector in~\eqref{eq:op-norm-def}. 
To this end, let $\mathbb{P}_0$ be any probability
measure on $\prod_{n} \R^{n \times n}$, such that its
projection on $\R^{n\times n}$ has the same law as   $A_n$.
The following theorem quantifies the proximity of the maximizing vector to $\1$.
Theorem~\ref{thm:maximizer-vector} is proved at the end of Section~\ref{subs-verification}.
An analogue of Theorem~\ref{thm:maximizer-vector} will later be proved for general  deterministic sequence of matrices (see Proposition~\ref{prop:vec-close}).
\blue{For a sequence of events $(E_n)_{n\geq 1}$ with $E_n$ being an event involving $A_n$, we say that $(E_n)_{n\geq 1}$ occurs $\PR_0$-\emph{eventually almost surely} if $E_n$ occurs for all large enough $n$, $\PR_0$-almost surely.}

\begin{theorem}\label{thm:maximizer-vector}  
Suppose  Assumption~\ref{assumption-1} holds. Also, 
let
\begin{equation}
    \label{def-maxvecn}
\vv_n := \argmax_{\xx \in\R^n: \|\xx\|_r \leq 1} \|A_n\xx\|_p
\end{equation}
and $\1$ denote the $n$-dimensional vector of all ones. 
Then the following hold: 
\begin{enumerate}[{\normalfont (a)}]
    \item For $1< p < r< \infty$, 
\begin{eq}\label{eq:largest-e-vector-pnr}
\|\vv_n - n^{-1/r} \1 \|_{\infty} 
\leq \frac{6p}{r-p} 
n^{-\frac{1}{r}}\sqrt{\frac{\log n}{n\mu_n}\times \frac{\sigma_n^2}{\mu_n}},
\quad \PR_0 \text{ eventually almost surely.}
\end{eq}

\item For $p =  r\in (1, \infty)$,
\begin{eq}\label{eq:largest-e-vector}
\|\vv_n - n^{-1/r} \1 \|_{\infty} 
\leq \frac{60r}{r-1}n^{-\frac{1}{r}}\sqrt{\frac{\log n}{n\mu_n}\times\frac{\sigma_n^2}{\mu_n}}, \quad \PR_0 \text{ eventually almost surely.}
\end{eq}
\end{enumerate}
\end{theorem}

\begin{remark}\normalfont
\blue{We will see in Section~\ref{sec:vector} that the vector bound for the $p<r$ case holds when $A_n^TA_n$ is irreducible and $A_n$ has concentrated row sums. 
These two properties, and hence the result in~\eqref{eq:largest-e-vector-pnr} is established (in Proposition~\ref{prop:vec-close}) under a weaker set of assumptions than Assumption~\ref{assumption-1}.}
\end{remark}

\subsection{Matrices with inhomogeneous variance profile}\label{ssec:inhomogeneous}

We now consider random matrices having an inhomogeneous variance profile.
In this case, to prove the asymptotic normality result we need the
matrix to be dense (i.e., the matrix entries have asymptotically non-vanishing mean and variance). 
This is because our proof uses an upper bound on the second largest eigenvalue of the matrix, recently established in~\cite{AEK19}, which requires the matrix to be dense.
The $\ell_\infty$-approximation of the 
maximizing vector, however, still holds for 
analogous sparse matrices.

We start by stating the set of assumptions on the sequence of random matrices that are needed for the $\ell_\infty$-approximation of the maximizing vector.
\begin{assumption} \normalfont \label{assumption-inhom}
For each fixed $n\geq 1$, let $A_n = (\aij)_{i,j = 1}^n$ be a symmetric
random matrix such that
\begin{enumerate}[(i)]
\item\label{assumption2-i} $(\aij)_{i,j=1,i<j}^n$ is a collection of independent random variables with $a_{ij}$ having distribution $F_{ij}^n$ supported on $[0,\infty)$, mean $\mu_n$ and variance $\sigma_n^2(i,j)$. 
Also, $\aii =0$ for all $i\in [n]$.
\item\label{assumption2-ii} There exists a sequence $(\bar{\sigma}_n)_{n\in\N} \subset (0,\infty)$, and constants $c_*,c^* \in (0,\infty)$
  such that
  $$c_*\leq \liminf_{n\to\infty}\min_{1\leq i<j\leq n}\frac{\sigma_n(i,j)}{\bar{\sigma}_n}\leq  
 \limsup_{n\to\infty}\max_{1\leq i<j\leq n}\frac{\sigma_n(i,j)}{\bar{\sigma}_n} \leq c^*.$$
\item \label{assumption2-iii} \blue{$\mu_n$ and $\bar{\sigma}_n$ satisfies Assumption~\ref{assumption-1}~\eqref{assumption1-ii} by replacing $\sigma_n$ by $\bar{\sigma}_n$.}
 \item \label{assumption2-v} There exists $c>0$, such that 
\begin{equation}\label{eq:sigma-lower-bound-inhom}
    \max_{1\leq i<j\leq n} 
    \E\big[|a_{ij}^n-\mu_n|^k\big]\leq \frac{k!}{2} c^{k-2}\bar{\sigma}_n^2\quad \text{for all }k\geq 3.
\end{equation}
\end{enumerate}
\end{assumption} 
\begin{theorem}\label{thm:maximizer-vector-inhom}  
Suppose $A_n$ is a symmetric random matrix satisfying \textrm{Assumption~\ref{assumption-inhom}}. 
Also, as in~\eqref{def-maxvecn}, recall that
\begin{equation}
\vv_n := \argmax_{\xx \in\R^n: \|\xx\|_r \leq 1} \|A_n\xx\|_p.  
\end{equation}
Then $\vv_n$ satisfies the same approximations as in~\eqref{eq:largest-e-vector-pnr} and~\eqref{eq:largest-e-vector}, but with $\sigma_n$ replaced by $\bar{\sigma}_n$.
\end{theorem}
Theorem~\ref{thm:maximizer-vector-inhom} is proved at the end of Section~\ref{subs-verification}.
Next, we state the asymptotic normality result.

\begin{theorem}
\label{thm:asymptotic-normality-inhom}

Fix any $1< p\leq r< \infty$.  
Consider the sequence of random matrices $(A_n)_{n\in\N}$ satisfying {Assumption~\ref{assumption-inhom}} and define 
$
    \bA_n:= n^{-(\frac{1}{p}-\frac{1}{r})} A_n.
$
Also assume that 
$
    \liminf_{n\to\infty} \bar{\sigma}_n>0.
$
Then as $n\to\infty$,
 \begin{eq}\label{eq:asymp-normal-inhom}
\frac{n^2}{2\sqrt{\sum_{i<j}\sigma_n^2(i,j)}}\big(\opera{\bA_n}- \alpha_n(p,r)\big) \dto Z \sim \mathrm{Normal}(0,2),
\end{eq} where 
\begin{equation}\label{eq:centering-inhom}
    \alpha_n(p,r)= (n-1)\mu_n  + \Big(p-1+\frac{1}{r-1}\Big)\frac{\sum_{i<j}\sigma_n^2(i,j)}{n^2\mu_n}.
\end{equation} 
\end{theorem}
Theorem~\ref{thm:asymptotic-normality-inhom} is proved in Section~\ref{ssec:asymp-normality}.
\\

Similar to Remark~\ref{rem:diag}, the zero diagonal entry is not a strict requirement in
Theorem~\ref{thm:asymptotic-normality-inhom}. 
The expression of $\alpha_n(p,r)$ in~\eqref{eq:centering-inhom} can be suitably updated to accommodate 
nonnegative random diagonal entries.

\subsection{Special cases}\label{ssec:special}
\paragraph*{Adjacency matrices of \texorpdfstring{Erd\H{o}s-R\'enyi}{Erdos-Renyi} random graphs.}
Let $\mathrm{ER}_n(\mu_n)$ denote an Erd\H{o}s-R\'enyi random graph with $n$ vertices and connection probability $\mu_n$.
As an immediate corollary to Theorems~\ref{thm:asymptotic-normality} and \ref{thm:maximizer-vector}, we obtain the asymptotic normality for adjacency matrices of certain sequences of $\mathrm{ER}_n(\mu_n)$ graphs.

\begin{corollary}\label{cor:errg-asymp-norm}
Fix any $1< p\leq r< \infty$ and 
let $A_n$ denote the adjacency matrix of $\mathrm{ER}_n(\mu_n)$.  
For \blue{$\mu_n = \omega( n^{-\frac{1}{3}}\log^{2/3}n )$,} the vector bounds in \eqref{eq:largest-e-vector-pnr} and \eqref{eq:largest-e-vector}, and the asymptotic normality result in \eqref{eq:asymp-normal} hold with $\sigma_n^2 = \mu_n(1-\mu_n)$.
\end{corollary}

\paragraph*{Grothendieck's \texorpdfstring{$\ell_r$}{lr}-problem.}
We now investigate the behavior of the $\ell_r$ quadratic maximization problem, also known as the $\ell_r$ Grothendieck problem.
For any $n\times n$ matrix $A_n$, the $\ell_r$ Grothendieck problem concerns the solution to the following quadratic maximization problem. 
For $r\geq 2$, define
\begin{equation}
    \label{eq:mr-def}
    M_r(A_n):= \sup_{\|\xx\|_{r}\leq 1} \xx^T A_n\xx.
\end{equation}
In general, finding $M_r(A_n)$ is NP-hard~\cite{KN12}.
However, in the case of a matrix $A$ with nonnegative entries, for which $A^TA$ is irreducible, Proposition~\ref{prop:groth} below states that the $\ell_r$ Grothendieck problem is a special case of the $r\to p$ norm problem. 
\begin{proposition}
\label{prop:groth}
\blue{Let $A$ be a symmetric matrix with nonnegative entries such that $A^TA$ is  irreducible.} 
Then for any $r\geq 2$, 
$
    M_r(A) = \|A\|_{r\to r^*},
$
where $r^*= r/(r-1)$ is the H\"{o}lder conjugate of $r$.
\end{proposition}
Proposition~\ref{prop:groth}
is proved at the end of Section~\ref{sec:Groth}.
Together with Theorem~\ref{thm:asymptotic-normality}, this immediately yields the limit theorem for 
$\bA_n:= n^{-(1-\frac{2}{r})} A_n$ 
stated in the corollary below.

\begin{corollary}
\label{cor:groth-clt}
Let $(A_n)_{n\in\N}$ be a sequence of random matrices satisfying the assumptions of {Theorem~\ref{thm:asymptotic-normality}}. 
Then for any fixed $r\in [2,\infty)$, as $n\to\infty$, the asymptotic normality result in \eqref{eq:asymp-normal} holds for $M_r(\bA_n)$ with $p = r^* = r/(r-1)$. 
\end{corollary}

\section{Proof outline}\label{sec:outline}
The proof of Theorem~\ref{thm:asymptotic-normality} consists of three major steps:

\paragraph*{Step 1: Approximating the maximizing  vector.}
The first step is to find a good approximation for a maximizing vector $\vv_n$ for $\opera{A_n}$, as 
defined in \eqref{def-maxvecn}.  
 As stated in Theorem~\ref{thm:maximizer-vector},  we can precisely characterize the  $\ell_\infty$  distance between $\vv_n$ and $n^{-1/r} \1$,  the scaled vector of all ones in $\R^n$.
In fact we work with a general \emph{deterministic} sequence of symmetric nonnegative matrices (see Proposition~\ref{prop:vec-close}). 
When $p<r$, the required $\ell_\infty$-bound follows whenever the row sums are approximately the same, which we call \emph{almost regularity} (see Definition~\ref{defn:almost-regular}). 
We actually have a short and elementary proof when $p<r$. 
\blue{The proof for the case $p=r$ is more complicated and 
requires that the entries of $A_n^TA_n$ be of order $n\mu_n^2$. 
We call the latter property, which is stated more precisely in Definition~\ref{defn:two-hop},   
\emph{well-connectedness}.}

\paragraph*{Step 2: Approximating the $r\to p$ norm.}
The next step is to construct a suitable approximation of $\|A_n\|_{r\to p}$. 
With the strong bound in Theorem~\ref{thm:maximizer-vector}, a natural choice would be to approximate $\|A_n\|_{r\to p}$ by $\| A_n n^{-1/r} \1\|_{p}$.
However, such an approximation turns out to be insufficient on the CLT-scale. 
To this end, we use a nonlinear power iteration for finding $r\to p$ norms, introduced by Boyd~\cite{Boyd74}. 
We start the power iteration from the vector $\vv^{(0)}_n:=n^{-1/r}\1$. 
We show that the rate of convergence of this power-method depends on the proximity of $\vv^{(0)}_n$ to $\vv_n$ (which we now have from Theorem~\ref{thm:maximizer-vector}), and the second largest eigenvalue of $A_n$ (for which we use existing results from \cite{LS18, EKYY13, AEK19}). 
Our ansatz is that after only one step of Boyd's nonlinear power iteration, we arrive at a suitable approximation of $\opern$. 
For any $k\geq 1$, $t\in \R$, and $\xx = (x_1,\ldots, x_n)$, define $
\psi_k(t):= |t|^{k-1}\sgn(t)$, and $\Psi_k(\xx)= (\psi_k(x_i))_{i=1}^n.$ 
Then we show that (see Proposition~\ref{prop:norm-approximation})  the quantity
    \begin{eq}
       \|A_n\|_{r\to p}\approx \eta(A_n) := \frac{\|A_n\Psi_{r^*}(A_n^T\Psi_p(A_n\1))\|_{p}}{\|\Psi_{r^*}(A_n^T\Psi_p(A_n\1))\|_r}, 
    \end{eq}
    where $r^* := r/(r-1)$ denotes the H\"{o}lder conjugate of $r$,
provides the required approximation to $\opern$.   
As in Step~1, we also first show this approximation for a deterministic sequence of matrices 
satisfying certain conditions, 
and then show that the random matrices we consider almost surely satisfy these conditions.

\paragraph*{Step 3: Establishing asymptotic normality.}
The final step is to prove the asymptotic normality of the sequence 
$\{\eta(A_n)\}_{n \in \N}$.  
This is a nonlinear function, and as it turns out, the state-of-the-art approaches to prove CLT do not apply directly in our case.
For that reason, we resort to an elementary approach using Taylor expansion to obtain the limit law.
Loosely speaking, we show that 
$$\eta(A_n) \approx n^{\frac{1}{p} - \frac{1}{r} -1}\sum_{i,j}\aij  + \frac{1}{2}\Big(p-1 + \frac{1}{r-1}\Big)n^{\frac{1}{p} - \frac{1}{r}}\sum_{i,j}(\aij-\mu)^2,$$
which after appropriate centering and scaling yields the CLT result as stated in Theorem~\ref{thm:asymptotic-normality}.

\section{Preliminaries}\label{sec:prelim}

\subsection{Boyd's nonlinear power method}
We start by introducing the nonlinear power iteration method and stating some preliminary known results, along with a rate of convergence result that will be crucial for our treatment.
The framework for nonlinear power iteration was first proposed by Boyd~\cite{Boyd74}.
It has also been used in~\cite{BV11} to obtain approximation algorithms for the $r\to p$ norm of matrices with strictly positive entries.

 Henceforth, we fix $n \in \N$, and for notational simplicity, omit the subscript $n$, for example, using $A$ to denote $A_n$, etc. 
Let
$A$ be an $n\times n$ matrix with nonnegative entries.
For any $\xx\ne \bzero$, define the function 
$f(\xx):= \|A\xx\|_{p}/\|\xx\|_r,$ and set 
 $\newsigma := \sup_{\xx\ne 0}f(\xx)$.
If a vector $\vv$ is a local maximum (or, more generally,  critical point) of the function $f$, then since $f$ is smooth, the gradient of $f$ must vanish at that point.   
This critical point can further be written as the solution to a fixed point equation.
Now, if there is a unique positive  critical point,  the fixed point equation may potentially be used to construct an iteration that converges to the maximum, starting from a suitable positive vector. 
In fact, under suitable assumptions, this convergence can be proved to be geometrically fast.
The above description is briefly formalized below.
For $q > 1$, $t \in \R$ and $\xx \in \R^n$, define 
\begin{eq}
\label{eq:psi-def}
\psi_q(t):= |t|^{q-1}\sgn(t),\qquad \Psi_q(\xx) := (\psi_q(x_i))_{i=1}^n, 
\end{eq}
where $\sgn(t) = -1, 1,$ and $0$, for 
$t <0$, $t >0$, and $t = 0$, respectively.
Taking the partial derivative of $f$ with respect to $x_i$, we obtain, for $\xx \neq \bzero$, 
\begin{align}\label{eq:pder-f}
    \frac{\partial f(\xx)}{\partial x_i} = \|\xx\|_r^{-2}\Big[\|A\xx\|_p^{-(p-1)}\inner{\Psi_{p}(A\xx), A_i^T}\|\xx\|_r - \|\xx\|_r^{-(r-1)} \psi_r(x_i)\|A\xx\|_p\Big],
\end{align}
where $A_i$ denotes the $i$-th column of $A$.
Equating~\eqref{eq:pder-f} to zero for $i = 1, \ldots, n$,  yields
\begin{eq}\label{eq:equating-0}
 \|\xx\|_r^rA^T\Psi_{p}(A\xx) = \|A\xx\|_p^p\Psi_{r}(\xx) .
\end{eq}
Now, let $\uu$ with $\|\uu\|_r =1$ be a (normalized) solution to  \eqref{eq:equating-0} and set $\newsigma(\uu) := \|A\uu\|_{p}$.   
Then straightforward algebraic manipulations show that 
\begin{eq}\label{eq:fixed-point-S}
\Psi_{r^*}(A^T\Psi_p(A\uu)) = \big(\newsigma(\uu)\big)^{p(r^*-1)}\uu ,
\end{eq}
where recall that $r^* = r/(r-1)$.
We denote the operator arising on the left-hand side ~of~\eqref{eq:fixed-point-S} as follows:   
\begin{eq}\label{eq:boyditeration}
S\xx := \Psi_{r^*}(A^T\Psi_p(A\xx)), \quad 
W\xx := \frac{S\xx}{\|S\xx\|_r} \qquad \text{for }  \xx\ne \bzero.
\end{eq}
Then \eqref{eq:fixed-point-S} implies  
\begin{eq}\label{identities-SW}
    S\uu = \big(\newsigma(\uu)\big)^{p(r^*-1)} \uu, \quad W\uu = \uu,
\end{eq}
where the last equality uses the fact that $\|\uu\|_r = 1.$ 
Thus, any solution to \eqref{eq:fixed-point-S} is a fixed point of the operator $W$.
The following lemma proves uniqueness of this fixed point among all nonnegative vectors,  which can be viewed as a generalization of the classical Perron-Frobenius theorem.
The uniqueness in Lemma~\ref{lem:vcomp} was established for matrices with strictly positive entries in~\cite[Lemma 3.4]{BV11}. 
Below we show that their proof can be adapted to matrices with nonnegative entries when $A^TA$ is irreducible. 
\begin{lemma}
    \label{lem:vcomp}
    Assume that $A^TA$ is irreducible.
    Then \eqref{eq:fixed-point-S} has a unique solution $\vv$ among the set of all nonnegative vectors. 
    Further, $\vv$ has all positive entries.
\end{lemma}
\begin{proof}
\blue{First note that the maximizer of $\|A\xx\|_{p}/\|\xx\|_r$ over $\xx\ne \bzero$ (which always exists) satisfies~\eqref{eq:fixed-point-S}.
Also, all entries of such a maximizer are nonnegative.
To see this, if $\xx$ has a negative entry, then the value of $\|A\xx\|_{p}$ can be strictly increased by replacing the negative entry by its absolute value, without changing $\|\xx\|_r$.}

Next, we show that, when $A^TA$ is irreducible, any non-zero, nonnegative vector satisfying~\eqref{eq:fixed-point-S} must have strictly positive entries. 
This, in particular, will also prove that $\vv$ has all positive entries. 
We argue by contradiction.  
Let $\xx$ be a non-zero, nonnegative vector satisfying \eqref{eq:fixed-point-S} and 
suppose, $i \in [n]$ be such that  $x_i=0$. 
Then, by \eqref{eq:fixed-point-S} and \eqref{identities-SW} we have 
\begin{eq}
    (S\xx)_i = 0
    &\implies (A^T (\Psi_p (A\xx))_i = \sum_{j=1}^n a_{ji} \left|\sum_{k=1}^n a_{jk} x_k\right|^{p-1} = 0, 
\end{eq}
In fact, we have 
\begin{eq}\label{eq:Sx0}
    (S\xx)_i = 0
    \implies (A^TA\xx)_i=\sum_{j=1}^n a_{ji}\Big(\sum_{k=1}^n a_{jk}x_k\Big) =0,
\end{eq}
since all the elements of $A$ and $\xx$ are nonnegative, if $A^T\Psi_p(A\xx)=0$, then $\Psi_{2}(A^T\Psi_2(A\xx))=0$ as well. 
Observe that~\eqref{eq:Sx0} implies $x_j=0$ for all $j\in [n]$ for which there exists $j'\in [n]$ with $a_{j'i}>0$ and  $a_{j'j}>0$.
Repeating the above with $i$ replaced by any such $j$,
we conclude that $x_j = 0$. 
Continuing in this way 
and using the irreducibility of $A^TA$, it follows that $x_j=0$ for all $j=1,\ldots,n$, which this leads to a contradiction. 
Thus, $\xx$ must have strictly positive entries.

To show uniqueness, let $\uu\neq \vv$ be two nonnegative non-zero vectors satisfying~\eqref{eq:fixed-point-S} with  $\|\uu\|_r=\|\vv\|_r=1$. 
Further, without loss of generality, assume that $\|A\uu\|_p\leq \|A\vv\|_p$.
By the above argument, both $\uu$ and $\vv$ have all positive entries. 
Then there must exist $\theta\in (0,1]$ such that $\uu - \theta \vv$ has a zero coordinate. 
Let $\theta$ be the smallest such number.
Define $U := \{k: u_k- \theta v_k =0\}$, and note that $u_j- \theta v_j>0$ for all $j\in U^c$.
Since $\|\uu\|_r = \|\vv\|_r$ and $\uu\neq \vv$, it follows that $U^c\neq \varnothing$. 
\begin{claim}
There exists $k\in U$ such that 
\begin{eq}\label{strict-ineq-unique}
     (S\uu)_k > (S\theta \vv)_k = \theta^{\frac{p-1}{r-1}} (S \vv)_k.
\end{eq}
\end{claim}
\begin{claimproof}
First, note that since $A^TA$ is irreducible,
there exists $k_1\in U$, $k_2\in [n]$, and $k_3\in U^c$, such that both
$a_{k_1k_2}$ and $a_{k_2k_3}$ are positive.
Therefore, the inequalities $u_{k_3}> \theta v_{k_3}$, $a_{k_2k_3}>0$, $u_i\geq \theta v_i$ for all $i\in [n]$ (the latter holds by the minimality of $\theta$), and the nonnegativity of $A$, $u$, and $v$ yield
\begin{eq}\label{eq:local-4.8}
    \big(\Psi_p(A\uu)\big)_{k_2}>\big(\Psi_p(A(\theta\vv))\big)_{k_2}\quad \text{and}\quad
    \big(\Psi_p(A\uu)\big)_{i}\geq\big(\Psi_p(A(\theta\vv))\big)_{i}\  \text{for all}\  i\in [n].
\end{eq}
This, together with the fact that $a_{k_1k_2}>0$, implies 
$(A^T\Psi_p(A\uu))_{k_1}>(A^T\Psi_p(A(\theta\vv)))_{k_1},$
 and by \eqref{eq:boyditeration}, \eqref{strict-ineq-unique} holds with $k=k_1$. 
\end{claimproof}
Now fix some $k\in U$ satisfying \eqref{strict-ineq-unique}. 
Then, using \eqref{eq:fixed-point-S}, one observes that
\begin{eq}
     \gamma(\uu)^{p} =  \frac{(S\uu)_k^{r-1}}{u_k^{r-1}}>\frac{\theta^{p-1}(S\vv)_k^{r-1}}{(\theta v_k)^{r-1}} = \theta^{p-r} \gamma(\vv)^{p}. 
\end{eq}
Since $p\leq r$ and $\theta\in (0, 1]$, this yields $\|A\uu\|_p = \gamma(\uu)>\gamma(\vv)=\|A\vv\|_p$, which contradicts the initial assumption that $\|A\uu\|_p\leq \|A\vv\|_p$. 
This proves the uniqueness. 
\end{proof}

The (nonlinear) power iteration for finding $\gamma$ consists of the following iterative method: 
Let $\vv^{(0)}$ be a vector with positive entries and $\|\vv^{(0)}\|_r=1$.
Then for $k\geq 0$, define 
\begin{eq}\label{eq:iteration-algo}
\vv^{(k+1)} := W \vv^{(k)}. 
\end{eq}
In general, the above iteration may not converge to the global maximum $\gamma$.
However, as the following result states, if in addition to having nonnegative entries, the matrix $A^TA$ is irreducible, then the  iteration must converge to the unique positive fixed point. 
\begin{proposition}[{\cite[Theorem 2]{Boyd74}}] \label{prop:iteration-converge}
Fix any $1< p \leq r <\infty$. 
Let $A$ be a matrix with nonnegative entries such that  $A^TA$ is irreducible.
If $\vv^{(0)}$ has all positive entries,
then $\lim_{k\to\infty}\|A\vv^{(k)}\|_p = \newsigma$. 
\end{proposition}

\subsection{Rate of convergence}

Due to Lemma~\ref{lem:vcomp}, henceforth we will reserve the notation $\vv$ to denote the unique maximizer   in~\eqref{eq:op-norm-def} having positive entries and $\|\vv\|_r=1$. The notation $\newsigma= \newsigma(\vv) = \|A \vv\|_p$ denotes the operator norm $\opera{A}$.
Next, we will study the rate of convergence of $\vv^{(k)}$ to~$\vv$. 
Specifically, we obtain a fast convergence rate once the approximating vector comes within a certain small neighborhood of the maximizing vector.
The rate of convergence result builds on the line of arguments used in the proof of \cite[Theorem 3]{Boyd74}.
However, as it turns out, since we are interested in the asymptotics in $n$, the rate obtained in~\cite{Boyd74} does not suffice (see in particular, \cite[Equation 16]{Boyd74}), and  we need the sharper result stated in Proposition~\ref{th:boydmain}.

Recall for any $\xx,\yy\in \R^n$, we write $\xx \oast \yy  = (x_iy_i)_i$.
Define the linear transformation
\begin{eq}\label{eq:defn:B}
B\xx&:=|\vv|^{2-r}\oast A^T(|A\vv|^{p-2}\oast (A\xx)),
\end{eq}
and the inner product 
\begin{eq}\label{eq:inner-prod-def}
[\xx,\yy]:= \langle |\vv|^{r-2}\oast \xx,\yy\rangle.
\end{eq}
When $A^TA$ is irreducible, $\vv$ has all positive entries by Lemma~\ref{lem:vcomp}, and thus \eqref{eq:defn:B} and \eqref{eq:inner-prod-def} are well-defined for all $p,r\geq 1$.  
Observe that this inner product induces a norm, which will henceforth be referred to as the ``$\vv$-norm'': 
\begin{eq}\label{eq:vnorm-def}
\|\xx\|_{\vv}:= [\xx,\xx]^{1/2} = \langle |\vv|^{r-2},|\xx|^2\rangle^{1/2}.
\end{eq} 
\blue{It is worthwhile to note that $\|\vv\|_{\vv}^2 = \|\vv\|_r^r$ and $[B\vv,\vv]^2 = \|A\vv\|_p^p$.}
The following fact is immediate. 
\begin{fact}\label{fact:sym-pd}
The operator $B$ is symmetric and positive semi-definite with respect to the inner product in \eqref{eq:inner-prod-def}.
\end{fact}
Fact~\ref{fact:sym-pd} implies that the eigenspace of $B$ has $n$ 
orthonormal basis vectors and $n$ nonnegative 
eigenvalues corresponding to the Rayleigh quotient
\begin{eq}\label{Rayleigh-coefficient-B}
\frac{[B\xx,\xx]}{[\xx,\xx]} = \frac{\langle |A\vv|^{p-2},|A\xx|^2\rangle}{\langle|\vv|^{r-2}, |\xx|^2\rangle}.
\end{eq}
Henceforth, we will refer to \eqref{Rayleigh-coefficient-B} as the $\vv$-Rayleigh quotient to emphasize the dependence on $\vv$.
Using \eqref{eq:fixed-point-S}, note that $B \vv  = \newsigma^p \vv$, and hence, $\newsigma^p$ is an eigenvalue of $B$.
Let $\lambda_2\geq \lambda_3 \geq \dots \geq \lambda_{n}$ be the other eigenvalues.
In fact, as shown in the proof of~\cite[Theorem~3]{Boyd74}, $\newsigma^p$ is the largest eigenvalue of $B$ and is simple.

Now, recall that the convergence rate of the 
classical (linear) power iteration for the largest eigenvalue of matrices depends on the the ratio between the largest and the second largest eigenvalues.
As it is stated in the proposition below, in the nonlinear case, this rate depends on the ratio of the largest and second largest eigenvalues of the operator~$B$.

\begin{proposition}\label{th:boydmain}
Let $A$ be an $n\times n$ matrix with nonnegative entries such that $A^TA$ is irreducible and $1< p \leq r <\infty$.  
Also let $\yy$ have all positive entries. 
There exists $\varepsilon_0 = \varepsilon_0 (p,r)>0$
and $C = C(p,r) >0$, both independent of $n$, such that if $\|\yy - \vv\|_{\infty} \leq \varepsilon$, then
\blue{\begin{eq}
     \|W\yy - \vv \|_{\vv} \leq (1+C\varepsilon)\frac{(p-1)\lambda_2}{(r-1)\newsigma^{p}} \|\yy - \vv \|_{\vv}.
\end{eq}}
Consequently, if for some $k \geq 1$ and $\varepsilon \leq \varepsilon_0$, $\vv^{(k)}$ has all positive entries and
$\|\vv^{(k)} - \vv\|_{\infty} \leq \varepsilon$, then
\begin{eq}
    \|\vv^{(k+1)} - \vv \|_{\vv} \leq (1+C\varepsilon)\frac{(p-1)\lambda_2}{(r-1)\newsigma^{p}} \|\vv^{(k)} - \vv \|_{\vv}.
\end{eq}
\end{proposition}

\begin{remark}\normalfont
It is worthwhile to point out that the convergence rate of the nonlinear power method depends on quantities in terms of the $\vv$-norm, which depends on the maximizer~$\vv$. Thus it might not be clear why this gives a useful rate of convergence. 
However, as we will see in Lemma~\ref{lem:ingredients},
the $\ell_\infty$-bound on the maximizing vector
in the nonlinear case, stated in Proposition~\ref{prop:vec-close}, enables us to 
obtain the desired rate of convergence result. 
\end{remark}

\begin{proof}[Proof of Proposition~\ref{th:boydmain}]
For any two fixed vectors $\xx,\hh \in \R^n$, and a function $f$, let us denote the directional derivative of $f$ at $\xx$ as
$$\delta f(\xx;\hh) := \lim_{\varepsilon\to 0} \frac{1}{\varepsilon} \big(f(\xx+ \varepsilon\hh ) - f(\xx)\big),$$ 
whenever the limit exists.
Recall that $\xx\oast \yy$ denotes the vector $(x_iy_i)_i$.   
Now, fix $1 < p \leq r < \infty.$  First, note that 
for a vector $\xx$ with all positive entries, $\delta  \Psi_p(\xx;\hh) = (p-1) \Psi_{p-1} (\xx) \oast \hh$, and therefore,
\begin{eq}
\label{eq:Sder}
\delta S(\xx; \hh) &= (r^*-1) \Psi_{r^*-1} (A^T\Psi_p A\xx) \oast \Big( A^T  \big((p-1)\Psi_{p-1}(A\xx) \oast A \hh \big)\Big) \\
&= \frac{p-1}{r-1} \Psi_0(A^T\Psi_p(A\xx)) \oast S\xx \oast L(\xx;\hh),
\end{eq} 
where $\Psi_0(\zz) = (1/z_i)_i$ for a vector $\zz$ with all positive entries and $L(\xx;\hh) := A^T  (\Psi_{p-1}(A\xx) \oast A \hh )$.
\blue{Here, due to the irreducibility of $A^TA$, note that $A^T\Psi_p(A\xx)$ has all positive entries whenever $\xx$ does.}
Also, for $g(\xx) := \|S\xx\|_r$,  using~\eqref{eq:boyditeration}
 and \eqref{eq:Sder}, we see that  
\begin{eq}\label{eq:gder}
     \delta g(\xx;h) &= \frac{1}{r} \frac{1}{\|S\xx\|_r^{r-1}} \binner{r\Psi_r (S\xx), \delta 
     S(\xx;  \hh)} \\
     & = \frac{p-1}{r-1} \frac{1}{\|S\xx\|_r^{r-1}} \binner{A^T \Psi_p(A\xx) , \Psi_0(A^T\Psi_p(A\xx)) \oast S\xx \oast L(\xx;\hh)} \\ 
     &= \frac{p-1}{r-1} \frac{1}{\|S\xx\|_r^{r-1}}\inner{S\xx,L(\xx;\hh)} = \frac{p-1}{r-1} \frac{1}{\|S\xx\|_r^r}\inner{W\xx,L(\xx;\hh)}.
\end{eq}
Now observe that since $W\xx \|S\xx\|_r = S\xx$,
\begin{eq}
    \delta W(\vv, \hh) \|S\vv\|_r + W(\vv)\delta g(\vv; \hh) = \delta S(\vv; \hh)
\end{eq}
Therefore, from \eqref{eq:Sder} and~\eqref{eq:gder} it follows that
\begin{eq}\label{eq:directional-der-W}
    \delta W(\vv;\hh) = \Big(\frac{p-1}{r-1}\Big) \frac{1}{\|S\vv\|_r^{r-1}} \big[|W\vv|^{2-r} \oast L(\vv;\hh) - W\vv \langle W\vv,L(\vv;\hh)\rangle\big],
\end{eq}
where we have used the fact that $\vv$ and $W\vv$ have nonnegative entries
Now,  $\delta W(\vv;\cdot)$ is a linear transformation.
Clearly, $\delta W(\vv;\vv) = 0$ since $L(\vv;\vv) = \Psi_r(S\vv)$. 
Further, it follows that the eigenvectors of $\delta W(\vv;\cdot)$ corresponding to the non-zero eigenvalues coincide with the eigenvectors of $B$ defined in \eqref{eq:defn:B} corresponding to $\lambda_2,\dots, \lambda_n$ given by \eqref{Rayleigh-coefficient-B}. 
This follows since $B\hh = \lambda \hh$ for some nonzero $\lambda\neq \newsigma$ implies that $L(\vv;\hh) = \lambda |\vv|^{r-2} \oast \hh$, which together with $W\vv \propto \vv$ yields that 
\begin{eq}
     \inner{W\vv, L(\vv;h)} \propto \inner{\vv,|\vv|^{r-2} \oast \hh} = [\vv,\hh] = 0.
\end{eq}
Thus the second term in \eqref{eq:directional-der-W} is zero. Also the first term in \eqref{eq:directional-der-W} is proportional to $\vv$, which yields the equality of the eigenvectors.
In fact, the eigenvalues of $\delta W(\vv;\cdot)$ are given by  $\frac{p-1}{r-1} \newsigma^{-p}\lambda_i$.
Since the Rayleigh coefficients in \eqref{Rayleigh-coefficient-B} are computed with respect to the $\|\cdot \|_{\vv}$ norm, we have 
\begin{eq}\label{eq:directional-v-h}
    \|\delta W(\vv;\hh)\|_{\vv} \leq \frac{(p-1)\lambda_2}{(r-1)\newsigma^{p}} \|\hh\|_{\vv}.
\end{eq}

Now, for $t\in [0,1]$, define $\yy_t = \vv +t(\yy - \vv)$.
Note that $\yy_t$ has all positive entries, since $\vv$ has possitive entries, and $\yy$.
Thus, the same expression as \eqref{eq:directional-der-W} holds for $\delta W(\yy_t;\hh)$, with $\vv$ replace by $\yy_t$. 
Now, $\|\yy_t - \vv \|_{\infty} \leq \|\yy - \vv \|_{\infty} \leq \varepsilon$, for any $t\in [0,1]$. 
Using the fact that $(1+\varepsilon)^a = 1+O(\varepsilon)$, it follows that there exists a constant $C<\infty$ and $\ve_0 > 0$ 
both depending only on $p,r$, such that for all $\varepsilon\leq \ve_0$,
\begin{eq}\label{eq:directional-der-bound-y}
    \delta W(\yy_t;\hh) \leq  (1+ C \varepsilon)\delta W(\vv;\hh).
\end{eq}
Now, observe that 
$$\delta W(\yy_t; \yy- \vv) = \frac{\dif}{\dif t}(W\yy_t).$$
and therefore, using \eqref{identities-SW} and the fact that $\yy_0 = \vv$ and $\yy_1 = \yy$, we obtain
\begin{eq}
    W \yy - \vv =  W \yy - W\vv = \int_0^1 \delta W(\yy_t; \yy- \vv) \dif t.
\end{eq}
Thus, \eqref{eq:directional-v-h} and \eqref{eq:directional-der-bound-y} implies that 
\begin{eq}
    \|W\yy - \vv\|_{\vv} \leq  (1+ C \varepsilon) \frac{(p-1)\lambda_2}{(r-1)\newsigma^{p}} \|\yy - \vv\|_{\vv},
\end{eq}and the proof follows.
\end{proof}

\section{An $\ell_\infty$-approximation of the maximizer}
\label{sec:vector}
Given an $n \times n$ nonnegative 
matrix  $A_n = (\aij)$ and $V\subseteq [n]$,  we write 
\begin{equation} 
\label{eiv}
d_n(i,V) := \sum_{j\in V} \aij, \quad i = 1, \ldots, n. 
\end{equation}
Also, we simply write $\degreei = d_n(i,[n])$.
When $A_n$ is the adjacency matrix of a graph on $n$ vertices, $\degreei$ represents the (out)-degree of vertex $i$.
\begin{defn}[Almost regular] \label{defn:almost-regular}
\normalfont
A sequence of matrices $(A_n)_{n\in \N}$ is called $(\varepsilon_n,\mu_n)_{n\in \N}$ \emph{almost regular} if there exists an $n_0\geq 1$ such that for all $n\geq n_0$
\begin{equation}
    \label{assumption:degree}
    \max_{i\in[n]} \big| \degreei - n\mu_n \big|\leq n\mu_n\varepsilon_n.
\end{equation}
\end{defn}
In order to show the proximity of the maximizing vector to $n^{-1/r}\1$ for the $p=r$ case, we need another asymptotic property in addition to the almost regularity defined above.  
\blue{\begin{defn}[Well-connected]
\label{defn:two-hop}\normalfont
For a constant $C^*\in (0, \infty)$, a sequence of matrices $(A_n)_{n\in \N}$ is called $(C^*, \mu_n)_{n\in \N}$ \emph{well-connected} if there exists an $n_0\geq 1$,  such that for all $n\geq n_0$ and $i, j\in [n]$, $\sum_{k\in [n]}\aik \akj \geq C^* n\mu_n^2$.
\end{defn}
When $A_n$ is an adjacency matrix, the  well-connected property ensures that there are sufficiently many 2-hop paths between any two sets of vertices.}
We now state the main result of this section:
\begin{proposition}\label{prop:vec-close} 
Let $(A_n)_{n\in \N}$ be a sequence of symmetric  matrices with nonnegative entries, such that $A_n^TA_n$ is irreducible for all $n\in \N$.
Assume that there exists 
$(\varepsilon_n)_{n \in \N} \subset (0,\infty)$ with $\varepsilon_n\to 0$, and 
$(\mu_n)_{n \in \N} \subset (0,1)$, 
such that  
$(A_n)_{n\in \N}$ is $(\varepsilon_n,\mu_n)_{n\in \N}$ almost regular. 
For each $n \in \N$, let $\vv_n$ be the maximizing vector for $\opera{A_n}$, 
as defined in~\eqref{def-maxvecn}. 
Then there exists an $n_0\geq 1$, such that the following hold: 
\begin{enumerate}[{\normalfont (a)}]
    \item For $1<p<r<\infty$, and for all $n\geq n_0$,
\begin{eq}\label{vec-close-p-neq-r}
\|\vv_n - n^{-1/r} \1 \|_{\infty} \leq  \frac{2p}{r-p}n^{-\frac{1}{r}}(\varepsilon_n + O(\varepsilon_n^2)).
\end{eq}
\item For $p=r \in (1,\infty)$, further assume that $(A_n)_{n\in \N}$ \blue{is $(C^*, \mu_n)_{n\in \N}$ well-connected for some constant $C^*>0$. 
Then for all $n\geq n_0$, 
\begin{eq}\label{vec-close-p-eq-r}
\|\vv_n - n^{-1/r} \1 \|_{\infty} \leq  \frac{10r}{C^*(r-1)} \varepsilon_n n^{-\frac{1}{r}} .
\end{eq}}
\end{enumerate}
\end{proposition}

 We prove Proposition~\ref{prop:vec-close}~(a) and~(b) in Sections~\ref{subs-pfcasea} and \ref{subs-pfcaseb}, respectively.

\subsection{Maximizer for the case \texorpdfstring{$p<r$}{p<r}}
\label{subs-pfcasea}
Given a maximizing vector $\vv_n$ for $\opera{A_n}$ as in \eqref{def-maxvecn}, define 
\begin{equation}
    \label{minmax}
 \minn := \min_{i= 1, \ldots, n} v_{n,i}, 
\qquad \mbox{ and } \qquad 
 \Maxn := \max_{i=1,\ldots, n} v_{n,i}. 
\end{equation}
Let $(\varepsilon_n, \mu_n)_{n\in \N}$ with $\varepsilon_n\to 0$ be as in the statement of Proposition~\ref{prop:vec-close}.
Suppose we can show that, for all sufficiently large $n$, and for some $C  \in (0,\infty),$ 
\begin{eq} \label{eq2:lb-minmax-ratio}
 \frac{\minn}{\Maxn} \geq  1 -C\varepsilon_n + O( \varepsilon^2_n).
\end{eq}
Then, $1= \sum_i v_{n,i}^r \leq n \Maxn^r$, so that $\Maxn\geq n^{-1/r}$. 
Also, \eqref{eq2:lb-minmax-ratio} yields 
$$1= \sum_{i=1}^n v_{n,i}^r \geq n \minn^r\geq n \Maxn^r(1-rC\varepsilon_n + O(\varepsilon_n^2)).$$
Together, this shows that 
\[ 
\|\vv_n - n^{-1/r} \1 \|_{\infty} \leq  C  n^{-\frac{1}{r}} (\varepsilon_n + O(\varepsilon_n^2)).
\] 
Thus, to show Proposition~\ref{prop:vec-close}, it is enough to prove \eqref{eq2:lb-minmax-ratio}
with $C=\frac{2p}{r-p}$.  

Recall Definition~\ref{defn:almost-regular} and the associated notation in~\eqref{eiv}.
Using \eqref{identities-SW}, \eqref{eq:boyditeration}, and \eqref{eq:psi-def}, together with $r^*- 1 = 1/(r-1)$, and the fact that $A_n$ is 
nonnegative and symmetric, we can use \eqref{eiv} and  \eqref{assumption:degree} to conclude 
 that for any $j$,
\begin{eq}\label{ub-max-entry}
(S\vv_n)_{j} = \left(\Psi_{r^*} \left(A_n^T \Psi_p (A_n \vv_n)\right)\right)_{j} 
&= \left|\left(A_n^T \Psi_p \left(A_n \vv_n)\right)\right)_{j}\right|^{\frac{1}{r-1}} \\
&\leq \Big(\sum_{i=1}^n a^n_{i j} (\Maxn \degreei )^{p-1}\Big)^{\frac{1}{r-1}} \\
&\leq \Big(\sum_{i=1}^n a^n_{j i} (\Maxn n\mu_n  (1+\varepsilon_n) )^{p-1}\Big)^{\frac{1}{r-1}}\\
&\leq \big((\Maxn n\mu_n)^{p-1}(1+\varepsilon_n)^{p-1}n\mu_n (1+\varepsilon_n)  \big)^{\frac{1}{r-1}}\\
&\leq \big(\Maxn^{p-1}(n\mu_n)^{p}\big)^{\frac{1}{r-1}}(1+\varepsilon_n)^{\frac{p}{r-1}}\\
&= \Maxn^{\frac{p-1}{r-1}}(n\mu_n)^{\frac{p}{r-1}}\Big(1+ \frac{p}{r-1}\varepsilon_n + O(\varepsilon_n^2)\Big).
\end{eq}
A similar computation yields the following lower bound: For any $i$,
\begin{eq}\label{lb-min-entry}
(S\vv_n)_{i} \geq \minn^{\frac{p-1}{r-1}}(n\mu_n)^{\frac{p}{r-1}}\Big(1- \frac{p}{r-1}\varepsilon_n + O(\varepsilon_n^2)\Big).
\end{eq}
\blue{Now, take any $i_0$ and $j_0$ such that $\minn= v_{n,i_0}$ and $\Maxn = v_{n,j_0}$. } 
Since by \eqref{identities-SW}, $\vv_n$ satisfies $S\vv_n \propto \vv_n$, we must have $\frac{(S\vv_n)_{i_0}}{\minn}=\frac{(S\vv_n)_{j_0}}{\Maxn}$, and consequently, \eqref{ub-max-entry} with $j = j_0$ and \eqref{lb-min-entry} with $i=i_0$ together imply that
\begin{eq}
\Maxn^{\frac{p-1}{r-1}-1}\Big(1+ \frac{p}{r-1}\varepsilon_n + O(\varepsilon_n^2)\Big) \geq \minn^{\frac{p-1}{r-1}-1}\Big(1- \frac{p}{r-1}\varepsilon_n + O(\varepsilon_n^2)\Big),
\end{eq}
which in turn implies
\[  \Maxn^{\frac{p-r}{r-1}} \geq \minn^{\frac{p-r}{r-1}}\Big(1- \frac{2p}{r-1}\varepsilon_n + O(\varepsilon_n^2)\Big).
\]
Thus, using the fact that $1 < p < r$, we have  \begin{eq} 
      \Big(\frac{\minn}{\Maxn}\Big)^{\frac{r-p}{r-1}} &\geq \Big(1- \frac{2p}{r-1}\varepsilon_n + O(\varepsilon_n^2)\Big)  \quad 
     \implies   \quad 
     \frac{\minn}{\Maxn} \geq \Big(1- \frac{2p}{r-p}\varepsilon_n + O(\varepsilon_n^2)\Big). 
\end{eq}
This completes the proof of \eqref{eq2:lb-minmax-ratio} with $C = 2p/(r-p)$, and hence Proposition~\ref{prop:vec-close}(a) follows. 
\qed 

\subsection{Maximizer for the case \texorpdfstring{$p=r$}{p=r}}
\label{subs-pfcaseb}
We now prove   Proposition~\ref{prop:vec-close}(b), which entails  establishing the bound in \eqref{vec-close-p-eq-r}  under both the almost-regularity and well-connected conditions  on $(A_n)_{n \in \N}$. 
The basic idea again is to show that if a vector $\vv_n$
satisfies $S\vv_n\propto \vv_n$, then the ratio
of its maximum and minimum must be converging to 1 as $n\to\infty$.
However, when $p=r$, one can see that the exponents 
of $M_n$ and $m_n$ in equations~\eqref{ub-max-entry} and~\eqref{lb-min-entry} become zero, and consequently the method used in Section~\ref{subs-pfcasea} fails.
The key insight to deal with this issue is to 
define two sets of vertices: one consisting of all
vertex indices $i$ such that $v_{n,i}$ is suitably large, and the other with $v_{n,i}$'s suitably small.
\blue{Due to the well-connectedness property, we can ensure 
that each vertex from one of these sets \emph{must} be connected to a certain
number of vertices from the other set in 2-hop paths.}
In that case, we show that if $M_n/m_n$ is not close to 1, then the ratio $(S\vv)_i/v_i$ will be very different
for the vertices for which $v_i$ is minimum and maximum, respectively.
This leads to a contradiction.

For any $r\in [2, \infty)$, $r^*\in (1,2]$ and further, by \cite[Lemma 8]{KMW18} and the symmetry of $A_n$, 
$A_n$, $\|A_n\|_{r\to r} = \|A_n^T\|_{r^*\to r^*}= \|A_n\|_{r^*\to r^*}$. 
Thus, to study the asymptotics of $\|A_n\|_{r\to r}$, it suffices to consider the case $r\in (1,2]$.
Let $n_0 \in \N$ be the maximum of the $n_0$ specified in the definitions of the almost-regularity and well-connected conditions and fix $n \geq n_0$. Also,
as in the proof of Proposition~\ref{prop:vec-close}(a), define $\minn$ and $\Maxn$ as in \eqref{minmax}.
Note that it suffices to show that for \blue{$\Delta_n:= (M_n-m_n)/2$, 
\begin{eq}\label{suff-to-show-p-eq-r}
  \frac{\Delta_n}{M_n} \leq \frac{5r \varepsilon_n}{C^*(r-1)},
\end{eq}
which is just a restatement of \eqref{eq2:lb-minmax-ratio}.
To this end, define 
$V_n := \{i: v_{n,i}\geq M_n-\Delta_n\}$,
and note that $M_n-\Delta_n = m_n+\Delta_n$.

In the rest of the proof, we will obtain upper and lower bounds on each  coordinate of $S\vv_n = \Psi_{r^*}(A_n^T\Psi_r(A_n\vv_n))$.   
Using the definition of $V_n$, we have for each $k \in [n]$, 
        \begin{eq}\label{eq:5-15}
             (A_n\vv_n)_k &\leq M_n\sum_{j\in V_n} \akj + (M_n-\Delta_n)\sum_{j\notin V_n} \akj = M_n\sum_{j\in [n]} \akj -\Delta_n\sum_{j\notin V_n} \akj,
             \\
             (A_n\vv_n)_k &\geq (m_n+\Delta_n)\sum_{j\in V_n} \akj  + m_n\sum_{j\notin V_n} \akj = m_n\sum_{j\in [n]} \akj+ \Delta_n\sum_{j\in V_n} \akj. 
        \end{eq}
    Take any $i_0$ and $j_0$ such that $\minn= v_{n,i_0}$ and $\Maxn = v_{n,j_0}$. 
    We will use the following elementary fact: For all $l\in (0,1]$ and $x\in [0,1]$,
    \begin{eq}\label{eq:elementary-fact}
        (1-x)^{l} \leq 1-\frac{lx}{2} \quad \text{and} \quad (1+x)^{l} \geq 1+\frac{lx}{2}.
    \end{eq}
    Then, by~\eqref{eq:5-15}, \eqref{eq:psi-def}, the fact that $r-1\in (0,1]$ and \eqref{eq:elementary-fact}, we have
    \begin{eq}\label{eq:local-5.17}
        \frac{(A_n^T\Psi_r(A_n\vv_n))_{j_0}}{M_n^{r-1}} &\leq \frac{1}{M_n^{r-1}} \sum_{k\in [n]} a_{kj_0}^n \bigg[M_n\sum_{j\in [n]} \akj -\Delta_n\sum_{j\notin V_n} \akj\bigg]^{r-1} \\
        & = \sum_{k\in [n]} a_{kj_0}^n \bigg(\sum_{j\in [n]} \akj\bigg)^{r-1} \bigg[1 -\frac{\Delta_n}{M_n}\frac{\sum_{j\notin V_n} \akj}{\sum_{j\in [n]} \akj}\bigg]^{r-1} \\ 
        & \leq \sum_{k\in [n]} a_{kj_0}^n \bigg(\sum_{j\in [n]} \akj\bigg)^{r-1} \bigg[1 -\frac{r-1}{2}\frac{\Delta_n}{M_n}\frac{\sum_{j\notin V_n} \akj}{\sum_{j\in [n]} \akj}\bigg].
    \end{eq}
    Also, since $A_n$ in $(C^*, \mu_n)$ well-connected, Definition~\ref{defn:two-hop} and the symmetry of $A_n$ imply
   \begin{eq}\label{eq:local5.18}
       \sum_{j\notin V_n}\sum_{k\in [n]} a_{kj_0}^n a_{kj}^n 
       \geq C^* n\mu_n^2 (n - |V_n|),
   \end{eq}
   and similarly,
   \begin{eq}\label{eq:local5.19}
       \sum_{j\in V_n}\sum_{k\in [n]} a_{ki_0}^n a_{kj}^n \geq C^* n\mu_n^2 |V_n|.
   \end{eq}
    Using the $(\varepsilon_n, \mu_n)_{n\in\N}$ almost regularity of $A_n$ and substituting  \eqref{eq:local5.18} in~\eqref{eq:local-5.17},
    we obtain
    \begin{eq}\label{balanced-perturb-1}
         \frac{(A_n^T\Psi_r(A_n\vv_n))_{j_0}}{M_n^{r-1}}
         &\leq (n\mu_n(1+\varepsilon_n))^{r-1}\sum_{k\in [n]} a_{kj_0}^n  \bigg[1 -\frac{r-1}{2}\frac{\Delta_n}{M_n}\frac{\sum_{j\notin V_n} \akj}{\sum_{j\in [n]} \akj}\bigg]\\
        &\leq (n\mu_n(1+\varepsilon_n))^{r} - \frac{r-1}{2} \frac{\Delta_n}{M_n} (n\mu_n(1+\varepsilon_n))^{r-2} \sum_{j\notin V_n} \sum_{k\in [n]} a_{kj_0}^n \akj\\
        & \leq 
        (n\mu_n(1+\varepsilon_n))^{r} - \frac{C^*(r-1)}{2} \frac{\Delta_n}{M_n} (n\mu_n(1+\varepsilon_n))^{r-2} n\mu_n^2(n-|V_n|).
    \end{eq}
Similarly, using almost regularity and \eqref{eq:local5.19} we obtain
    \begin{eq}\label{balanced-perturb-2}
        \frac{(A_n^T\Psi_r(A_n\vv_n))_{i_0}}{m_n^{r-1}} &\geq \sum_{k\in [n]} a_{ki_0}^n \bigg(\sum_{j\in [n]} \akj\bigg)^{r-1} \bigg[1 +\frac{r-1}{2}\frac{\Delta_n}{m_n}\frac{\sum_{j\in V_n} \akj}{\sum_{j\in [n]} \akj}\bigg]\\
        &\geq \sum_{k\in [n]} a_{ki_0}^n \bigg(\sum_{j\in [n]} \akj\bigg)^{r-1} \bigg[1 +\frac{r-1}{2}\frac{\Delta_n}{M_n}\frac{\sum_{j\in V_n} \akj}{\sum_{j\in [n]} \akj}\bigg]\\
        & \geq 
        (n\mu_n(1-\varepsilon_n))^{r} + \frac{C^*(r-1)}{2} \frac{\Delta_n}{M_n} (n\mu_n)^{r-2}\frac{(1-\varepsilon_n)^{r-1}}{1+\varepsilon_n} n\mu_n^2|V_n|.
    \end{eq}
   Since $\vv_n$ satisfies $S\vv_n \propto \vv_n$, we must have $\frac{(A_n^T\Psi_r(A_n\vv_n))_{j_0}}{M_n^{r-1}} = \frac{(A_n^T\Psi_r(A_n\vv_n))_{i_0}}{m_n^{r-1}}$. 
   Thus, combining  \eqref{balanced-perturb-1} and \eqref{balanced-perturb-2}, we get for large enough $n$,
   \begin{eq}\label{eq:local-5.22}
       \frac{C^*(r-1)}{2} \frac{\Delta_n}{M_n} (n\mu_n)^{r-2}n\mu_n^2\Big[(1+\varepsilon_n)^{r-2}(n-|V_n|) 
       &+ \frac{(1-\varepsilon_n)^{r-1}}{1+\varepsilon_n} |V_n| \Big]\\
       &
       \leq (n\mu_n)^r\Big[(1+\varepsilon_n)^r - (1-\varepsilon_n)^r\Big]
   \end{eq}
   Next, using $\varepsilon_n \to 0$, \eqref{eq:elementary-fact} and the fact that $r\in (1, 2]$, we can lower bound the left-hand-side of \eqref{eq:local-5.22} as follows:
   \begin{eq}\label{eq:local-5.23}
       &\frac{C^*(r-1)}{2} \frac{\Delta_n}{M_n} (n\mu_n)^{r-2}n\mu_n^2\Big[(1+\varepsilon_n)^{r-2}(n-|V_n|) 
       + \frac{(1-\varepsilon_n)^{r-1}}{1+\varepsilon_n} |V_n| \Big]\\
       & \geq 
       \frac{C^*(r-1)}{2} \frac{\Delta_n}{M_n} (n\mu_n)^{r-2}n\mu_n^2\Big[
       n-|V_n| + (1 - 2r\varepsilon_n)|V_n| \Big]\\
       & = \frac{C^*(r-1)}{2} \frac{\Delta_n}{M_n} (n\mu_n)^{r-2}n\mu_n^2\big[
       n - 2rn\varepsilon_n \big]\\
       &\geq 
       \frac{C^*(r-1)}{2} \frac{\Delta_n}{M_n} (n\mu_n)^{r}\Big[
       1  - 2r\varepsilon_n\Big].
   \end{eq}
   Therefore, using \eqref{eq:local-5.23} and Definition~\ref{defn:two-hop} in \eqref{eq:local-5.22} shows that for large enough $n$,
   \begin{align*}
       \frac{\Delta_n}{M_n} \leq \frac{2}{C^*(r-1)(1- 2r\varepsilon_n)}(2r\varepsilon_n + o(\varepsilon_n))\leq \frac{5r \varepsilon_n}{C^*(r-1)}.
   \end{align*}
This proves  \eqref{suff-to-show-p-eq-r}, 
and hence, completes the proof of Proposition~\ref{prop:vec-close}~(b). \qed }

\section{Approximation of the maximizer for random matrices} 
\label{subs-verification}
In this section, we show that the assumptions in Proposition~\ref{prop:vec-close} are satisfied almost surely by the sequence of random matrices of interest. 
This will complete the proofs of Theorems~\ref{thm:maximizer-vector} and~\ref{thm:maximizer-vector-inhom}.
Let $\mathbb{P}_0$ be any probability
measure on $\prod_{n} \R^{n \times n}$, such that its
projection on $\R^{n\times n}$ has the same  law as $A_n$, as defined in Assumption~\ref{assumption-1}.

\subsection{Random matrices are almost regular and well-connected}

In Lemmas~\ref{lem-verification} and~\ref{lem-verification-inhom}, we verify
the almost regularity and well-connectedness conditions for the homogeneous and inhomogeneous instances of  the random
matrix sequences, respectively.

\begin{lemma}
\label{lem-verification}
Let $(A_n)_{n \in \N}$ be a sequence of  random matrices satisfying  Assumptions \ref{assumption-1}~\eqref{assumption1-i},~\eqref{assumption1-iv}.
Also, suppose that 
\blue{\begin{eq}\label{eq:eps-choice}
    \varepsilon_n =3\bigg(\frac{\log n}{n\mu_n}\times \frac{\sigma_n^2}{\mu_n}\bigg)^{1/2}. 
\end{eq}}
\begin{enumerate}
    \item \blue{Suppose that  $\sigma_n^2 \geq \frac{9c^2\log n}{2n}$, where $c$ is as in Assumption \ref{assumption-1}~\eqref{assumption1-iv}.} Then $(A_n)_{n \in \N}$ is $(\ve_n, \mu_n)_{n\in \N}$ almost regular, $\mathbb{P}_0$-almost surely. 
    \item \blue{If Assumption~\ref{assumption-1}~\eqref{assumption1-ii} is satisfied, then for any constant $C^*\in (0,1)$, $(A_n)_{n \in \N}$ is also $(C^*,  \mu_n)_{n\in \N}$ well-connected, $\mathbb{P}_0$-almost surely.}
\end{enumerate}
\end{lemma} 

\begin{proof} 
\emph{Verification of almost regularity.}
First, note that $\sum_{j\in [n]\setminus \{i\}}\E[(\aij - \mu_n)^2] \leq n\sigma_n^2$
and Assumption~\ref{assumption-1}~\eqref{assumption1-iv} provides the moment conditions required for Bernstein's inequality (see~\cite[Corollary 2.11]{BLM13}). 
 Therefore,    
 using the fact that $(\aij)_{i<j}$ are i.i.d.~as well as the union bound,  and then applying   \cite[Corollary 2.11]{BLM13} for both the upper and lower tails, we conclude that  for all sufficiently large $n$, 
\begin{eq}\label{eq:deg-sum}
\PR\big(\exists\ i: \big|\degreei-n\mu_n\big|>n\mu_n\varepsilon_n\big) 
&\leq n \PR\big( \big|\degreeone  -n\mu_n\big|>n\mu_n\varepsilon_n\big) \\
&\leq 2 n \exp\Big(- \frac{n^2\mu_n^2\varepsilon_n^2}{2(n\sigma_n^2+cn\mu_n\varepsilon_n)}\Big),
\end{eq}
\blue{where $c$ is as given in Assumption~\ref{assumption-1}~\eqref{assumption1-iv}. 
Since 
$$c n\mu_n \varepsilon_n = 3c \bigg(n^2\mu_n^2\times\frac{\log n}{n\mu_n}\times \frac{\sigma_n^2}{\mu_n}\bigg)^{1/2}
= 3c\sigma_n \sqrt{n\log n}\leq \frac{n\sigma_n^2}{2},$$
and
$\frac{n^2\mu_n^2\varepsilon_n^2}{3n\sigma_n^2}= 3\log n$,
this implies 
$$\PR\big(\exists\ i: \big|\degreei-n\mu_n\big|>n\mu_n\varepsilon_n\big) 
\leq \exp\big(- 3\log n + \log n \big) = n^{-2},$$
which is summable in $n$.
Thus the almost regularity holds $\PR_0$-almost surely due to the Borel-Cantelli lemma.\\

\noindent 
\emph{Verification of well-connectedness.}
Note that it suffices to prove the following claim.
\begin{claim}
\label{lem:2-hop-random}
Define the sequence $\big(\varepsilon_n'\big)_{n\geq 1}$ as
\begin{equation}\label{eq:eps'-def}
   \varepsilon_n':= \frac{\sigma_n^2}{\mu_n^3} \frac{(\log n)^2}{n}.
\end{equation}
Then 
for all $i, j\in [n]$, $\big|\sum_{k}a_{ik}a_{kj} - n\mu_n^2\big|\leq n\mu_n^2\sqrt{\varepsilon_n'}$, $\PR_0$-almost surely.
\end{claim}

\begin{claimproof}
First, note that $\varepsilon_n'\to 0$ as $n\to\infty$ since $\frac{\sigma_n^2}{\mu_n} = O(1)$, $\sqrt{n} \mu_n = \omega(\log n)$ by Assumption~\ref{assumption-1}~\eqref{assumption1-ii}.
Next, for each fixed $i, j\in [n]$, note that
\begin{align*}
    \expt{\sum_{k}a_{ik}a_{kj}} &= (n-2)\mu_n^2\qquad
    \expt{\sum_{k}a_{ik}^2a_{kj}^2}=
    (n-2)(\sigma_n^2 +\mu_n^2)^2.
\end{align*}
By \cite[Corollary 2.11]{BLM13}, under Assumption~\ref{assumption-1}, we have for all large enough $n$,
\begin{align*}
    \PR\Big(\big|\sum_{k}a_{ik}a_{kj} - n\mu_n^2\big|>n\mu_n^2\sqrt{\varepsilon_n'}\Big)\leq 2 \exp\Big[- \frac{n^2\mu_n^4 \varepsilon_n'}{2(n(\sigma_n^2 +\mu_n^2)^2 + c''n\mu_n^2\sqrt{\varepsilon_n'})}\Big],
\end{align*}
where $c''$ is a constant that depends only on the constant $c$ in Assumption~\ref{assumption-1}~\eqref{assumption1-iv}.
The proof of the claim is completed by observing that since $\varepsilon_n'\to 0$ as $n\to\infty$, and $\mu_n$ and $\sigma_n^2/\mu_n$ are upper bounded by some fixed finite positive constant $K$, we have for all large enough $n$,
\begin{align*}
    \frac{n^2\mu_n^4 \varepsilon_n'}{2(n(\sigma_n^2 +\mu_n^2)^2 + c''n\mu_n^2\sqrt{\varepsilon_n'})}
    \geq 
    \frac{n^2\mu_n^4}{2n(K\mu_n +\mu_n^2)^2} \frac{\sigma_n^2}{\mu_n^3} \frac{(\log n)^2}{n}
    \geq \frac{(\log n)^2\sigma_n^2}{8K^2\mu_n}.
\end{align*}
\end{claimproof}
This completes the verification of $\PR_0$-almost sure well-connectedness.}
\end{proof}

The next lemma states the version of Lemma~\ref{lem-verification} in the inhomogeneous variance case.

\begin{lemma}
\label{lem-verification-inhom}
Let $(A_n)_{n \in \N}$ be a sequence of  random matrices 
that satisfies {Assumption \ref{assumption-inhom}}.
Also, suppose that 
\blue{$\varepsilon_n =3\big(\frac{\log n}{n\mu_n}\times \frac{\bar{\sigma}_n^2}{\mu_n}\big)^{1/2}.$
Then $(A_n)_{n \in \N}$ is $(\ve_n, \mu_n)_{n\in \N}$ almost regular, $\mathbb{P}_0$-almost surely. 
Moreover, for any constant $C^*\in (0,1)$, it is also $(C^*,  \mu_n)_{n\in\N}$ well-connected $\mathbb{P}_0$-almost surely. }
\end{lemma} 
\begin{proof}[Proof of Lemma~\ref{lem-verification-inhom}]
The proof follows verbatim the
proof of Lemma~\ref{lem-verification} once $\sigma_n$ is replaced by $\bar{\sigma}_n$.
\end{proof}

\begin{proof}[Proofs of Theorems~\ref{thm:maximizer-vector} and~\ref{thm:maximizer-vector-inhom}]
Note that Claim~\ref{lem:2-hop-random} also implies that $A_n^TA_n$ is irreducible.
Thus,
Theorems~\ref{thm:maximizer-vector} and~\ref{thm:maximizer-vector-inhom} are immediate from Proposition~\ref{prop:vec-close}, and Lemmas~\ref{lem-verification} and \ref{lem-verification-inhom}, respectively.
\end{proof}

\section{Approximating the \texorpdfstring{$r\to p$}{r-to-p} norm}\label{sec:power-method} 
The purpose of this section is to  identify a good approximation for $\|A_n\|_{r\to p}$ that is sufficiently explicit.
We use the power iteration method described in  Section~\ref{sec:prelim} starting  with initial  vector $\vv_n^{(0)} = n^{-1/r}\1$. 
Then after one iteration, we get the vector $\vv_n^{(1)}$ which, by 
\eqref{eq:iteration-algo} and \eqref{eq:boyditeration}, is given explicitly by 
\begin{equation}\label{eq:v1-def}
    \vv_n^{(1)} = \frac{\Psi_{r^*}(A_n^T\Psi_p(A_n\1))}{\|\Psi_{r^*}(A_n^T\Psi_p(A_n\1))\|_r}.
\end{equation}
Then define the quantity
\begin{eq}
\label{def-nun}
        \eta_n(A_n) 
         := \| A_n  \vv_n^{(1)}\|_p
        =   \frac{\|A_n\Psi_{r^*}(A_n^T\Psi_p(A_n\1))\|_{p}}{\|\Psi_{r^*}(A_n^T\Psi_p(A_n\1))\|_r},
    \end{eq}
which will serve as an approximation for $\|A_n\|_{r\to p}$.
We prove the following estimate: 
\begin{proposition} \label{prop:norm-approximation}
Let $(A_n)_{n\in\N}$, $(\varepsilon_n)_{n\in\N}$, and $(\mu_n)_{n\in\N}$ satisfy the same conditions as those imposed in Proposition~\ref{prop:vec-close}.
Then there exists a constant $C \in (0,\infty)$ (possibly depending on $p$ and $r$)   such that for all sufficiently large $n$,
$$\big|\|A_n\vv_n \|_{p} - \eta_n(A_n)\big|\leq C\frac{\Lambda_2^2 (n)\varepsilon_n}{\mu_n^2 n^{\frac{3}{2}+\frac{1}{r}}}\|A_n\|_{2\to p},$$
where $\eta_n$ is defined as in \eqref{def-nun} and
\begin{equation}
    \label{def-Lambdatwo}
    \Lambda_2^2 (n) :=\max_{\xx:\langle \1, \xx\rangle=0,\ \xx\neq \boldsymbol{0}} \frac{\|A_n\xx\|_2^2}{\|\xx\|_2^2}.
\end{equation}
\end{proposition}
The rest of this section is organized as follows. First, we estimate the closeness of $\vv_n^{(1)}$ to $\vv_n$ in Proposition~\ref{prop:approx-vector}.
In particular, we show that under the assumptions of  Proposition~\ref{prop:norm-approximation} (equivalently, Proposition~\ref{prop:vec-close}), 
$\vv_n$ can be approximated well by $\vv_n^{(1)}$.
This is then used to approximate the operator norm and complete the proof of Proposition~\ref{prop:norm-approximation}.

\begin{proposition}\label{prop:approx-vector}
Assume that the conditions of  Proposition~\ref{prop:vec-close} are satisfied.
Recall the definition of the $\vv$-norm from \eqref{eq:vnorm-def}.
Then there exists a constant $ C_2 < \infty$, possibly depending on $p, r$, such that for all sufficiently large $n$,
$$
\|\vv_n - \vv_n^{(1)}\|_{\vv_n}\leq C_2\frac{\Lambda_2^2(n)\varepsilon_n}{n^2\mu_n^2}, 
$$
where $\Lambda_2(n)$ is as defined in \eqref{def-Lambdatwo}. 
\end{proposition}

The next lemma provides key ingredients for the proof of Proposition~\ref{prop:approx-vector}.

\begin{lemma} \label{lem:ingredients}
Assume that $(A_n)_{n \in \N}$ satisfies the conditions of {\rm Proposition~\ref{prop:vec-close}} and $1 < p \leq r < \infty.$ 
Then the following hold: 
    \begin{enumerate}[{\normalfont (a)}]
        \item \label{fact:sigma-order}
    $\lim_{n\to\infty} \mu_n^{-1} n^{-(1+\frac{1}{p}-\frac{1}{r})}\opera{A_n} = 1$; 
        \item \label{lem:max}
$\max_{\xx: \|\xx\|_{\vv_n}\leq 1}\|A_n\xx\|_p = (1+o(1)) n^{\frac{1}{2}- \frac{1}{r}}\max_{\xx: \|\xx\|_{2}\leq 1}\|A_n\xx\|_p; $
        \item  
        Let $\lambda_2(n)$ be the second largest eigenvalue corresponding to the $\vv$-Rayleigh quotient defined in~\eqref{Rayleigh-coefficient-B}. 
        Then 
        \label{lem:lambda2-bound} $$\lambda_2(n)\leq2\mu_n^{p-2}n^{\frac{p(r-1)}{r} - 1} \Lambda_2^2(n).$$
    \end{enumerate}
\end{lemma}
\begin{proof}
(a) By  Proposition~\ref{prop:vec-close} and the almost regularity condition in Definition~\ref{defn:almost-regular}, it follows that
\begin{align*}
    \opera{A_n} 
    = \|A_n\vv_n\|_p
    &= \|A_n \1 (n^{-1/r} + o(n^{-1/r}))\|_p\\
    &= \|(n\mu_n + o(n\mu_n)) (n^{-1/r} + o(n^{-1/r}))\1\|_p\\
    &= \mu_n n^{1-1/r+1/p} + o(\mu_n n^{1-1/r+1/p}), 
\end{align*}
from which the claim in (a) follows. 

\noindent
(b) By \eqref{eq:vnorm-def} and 
Proposition~\ref{prop:vec-close}, we have for all sufficiently large $n$ and $x \in \R^n,$ 
\begin{eq}\label{v-norm-2-norm-compare}
    \|\xx\|_{\vv_n} &=\left(\sum_{i=1}^n|v_{n,i}|^{r-2}|x_i|^2\right)^{\frac{1}{2}} 
    = n^{-\frac{r-2}{2r}} \|\xx\|_2 (1+ o(1)).
\end{eq}
This implies that
\begin{eq}
    \max_{\|\xx\|_{\vv_n} \leq 1} \|A_n\xx\|_p =
    \max_{ \xx \neq 0} \frac{\|A_n\xx\|_p}{\|\xx\|_{\vv_n}}
    &=
    \max_{ \xx \neq 0} \frac{\|A_n\xx\|_p(1+ o(1))}{n^{-\frac{r-2}{2r} } \|\xx\|_2}\\
&    = n^{\frac{r-2}{2r}} (1+ o(1))\max_{\|\xx\|_{2} \leq 1} \|A_n\xx\|_p,
\end{eq}
which proves (b).

\noindent
(c)
Recall the inner product defined in~\eqref{eq:inner-prod-def} and 
that $\newsigma^p$ is the largest eigenvalue of  $B$ obtained from the $\vv$-Rayleigh quotient~\eqref{Rayleigh-coefficient-B}.
Thus, by using the Courant-Fischer theorem~\cite[Corollary~III.1.2]{Bhatia97}, and further justifications given below, note that 
\begin{align*}
    \lambda_2(n) = \min_{\uu \neq 0} \max_{\xx:[\uu,\xx]=0}\frac{[B\xx,\xx]}{[\xx,\xx]}
    &\leq \max_{\xx:[|\vv_n|^{2-r} ,\xx]=0}\frac{[B\xx,\xx]}{[\xx,\xx]}\\ 
    &= \max_{\xx:\langle \1, \xx\rangle=0}  \frac{[B\xx,\xx]}{\|\xx\|_{\vv_n}^2}\\
    &\leq n^{1-\frac{2}{r}} \max_{\xx:\langle \1, \xx\rangle=0} \frac{\langle|A_n\vv_n|^{p-2}, |A_n\xx|^2\rangle}{\|\xx\|_{2}^2}\\
    &\leq 2 \mu_n^{p-2}n^{1-\frac{2}{r}+(1-\frac{1}{r})(p-2)}\max_{\xx:\langle \1, \xx\rangle=0} \frac{\|A_n\xx\|_2^2}{\|\xx\|_2^2}\\
    &\leq 2\mu_n^{p-2}n^{\frac{p(r-1)}{r} - 1} \Lambda_2^2(n),
\end{align*}
where the second equality follows since for any $\xx$, $[|\vv_n|^{2-r} ,\xx]=0$ if and only if $\langle \1, \xx\rangle=0$, and the second and third inequalities follow from \eqref{v-norm-2-norm-compare} and the almost regularity.
\end{proof}

Now we have all the ingredients to complete the proof of Proposition~\ref{prop:approx-vector}.
\begin{proof}[Proof of Proposition~\ref{prop:approx-vector}] 
Note that for all large enough $n$, $\|\vv_n\|_{\infty} \leq 2 n^{-1/r}$ by Proposition~\ref{prop:vec-close}.  Thus,   for any $\xx \in \R^n$ with $\|\xx\|_{\infty}\leq 1$, it follows that 
\begin{eq}\label{eq:vnorm-inftynorm}
    \|\xx\|_{\vv_n} = \Big(\sum_{i=1}^n|v_{n,i}|^{r-2}|x_i|^2\Big)^{1/2} 
    \leq 2^{1-\frac{2}{r}} n^{-\frac{1}{2}+\frac{1}{r}}\|\xx\|_2 
    \leq 2^{1-\frac{2}{r}}n^{\frac{1}{r}} \|\xx\|_{\infty}.
\end{eq}
Then the vectors $\vv_n^{(0)} = n^{-1/r}\one$ and $\vv_n^{(1)}$ from \eqref{eq:v1-def} in the nonlinear power iteration satisfy (as justified below)
\begin{align*}
\|\vv_n-\vv_n^{(1)}\|_{\vv_n}
&\leq (1+o(1))\frac{(p-1)\lambda_2}{(r-1)\opera{A_n}^p}\|\vv_n - n^{-1/r}\one\|_{\vv_n} \\
& \leq (1+o(1))\frac{2(p-1)}{r-1}\frac{\Lambda_2^2(n)}{n^2\mu_n^2} \|\vv_n - n^{-1/r}\one\|_{\vv_n} \\
& \leq (1+o(1)) \frac{2^{2-\frac{2}{r}}(p-1)}{r-1}\frac{\Lambda_2^2(n)}{n^{2-\frac{1}{r}}\mu_n^2} \|\vv_n - n^{-1/r}\one\|_{\infty}\\
&\leq C\frac{\Lambda_2^2(n)\varepsilon_n}{n^2\mu_n^2} 
\end{align*}
where 
the first inequality is due to Proposition~\ref{th:boydmain} and the fact that $\opera{A_n}^p = \newsigma^p$,
the second inequality is due to Lemma~\ref{lem:ingredients}~\eqref{fact:sigma-order} and  Lemma~\ref{lem:ingredients}~\eqref{lem:lambda2-bound}, and the third inequality is due to \eqref{eq:vnorm-inftynorm}.
Proposition~\ref{prop:approx-vector} then follows from an application of 
Proposition~\ref{prop:vec-close}.
\end{proof}

\begin{proof}[Proof of Proposition~\ref{prop:norm-approximation}]
Once again considering the vector $\vv^{(1)}$ in \eqref{eq:v1-def} obtained after the first step of the nonlinear power iteration and $\eta_n(A_n) = \|A_n \vv_n^{(1)}\|_{p},$ from \eqref{def-nun}, we have
\begin{align*}
\big|\|A_n\vv_n\|_p - \eta_n(A_n)\big|&\leq 
\big|\|A_n\vv_n\|_p - \|A_n\vv^{(1)}_n\|_p\big|\\
&\leq \|A_n\vv_n - A_n\vv_n^{(1)}\|_p \\
&\leq \|\vv_n - \vv_n^{(1)}\|_{\vv_n} \max_{\|\xx\|_{\vv_n} \leq 1} \|A_n\xx\|_p\\
&\leq \|\vv_n - \vv_n^{(1)}\|_{\vv_n} (1+o(1)) n^{\frac{1}{2}-\frac{1}{r}}\max_{\|\xx\|_{2} \leq 1} \|A_n\xx\|_p\\
&\leq  \|\vv_n - \vv_n^{(1)}\|_{\vv_n} (1+o(1)) n^{\frac{1}{2}-\frac{1}{r}} \|A_n\|_{2\to p},
\end{align*}
where the third inequality is due to Lemma~\ref{lem:ingredients} \eqref{lem:max}.
Proposition~\ref{prop:norm-approximation} then follows on using Proposition~\ref{prop:approx-vector} to bound $\|\vv_n - \vv_n^{(1)}\|_{\vv_n}$.
\end{proof}

\section{Asymptotic normality} \label{sec:asym-normality}
In this section we establish asymptotic normality of $\eta_n(A_n)$ when $A_n$ satisfies Assumption~\ref{assumption-1}. 
We start in Section \ref{subs-prelim} 
with some preliminary results.

\subsection{Almost-sure error bound on the CLT scale} 
\label{subs-prelim}

First, recalling the definition of $\Lambda_2(n)$ in \eqref{def-Lambdatwo}, we prove the following lemma. 
\begin{lemma}\label{claim:lambda2-bound}
Under {Assumption~\ref{assumption-1}} the following holds:
\begin{eq}\label{eq:2nd-evalue-order}
     \Lambda_2 (n) \leq 3\sqrt{n} \sigma_n +\mu_n, \quad \PR_0 \mbox{ eventually almost surely}. 
\end{eq}
\end{lemma}
For the proof, it wil lbe convenient to define the following centered version of $A_n$:
\begin{equation}\label{eq:defn-A-n-0}
    A_n^0 := A_n - \mu_n \1\1^T + \mu_n I_n.
\end{equation}

\begin{proof}[Proof of Lemma~\ref{claim:lambda2-bound}]
First observe that for all vectors $\xx$ with $\langle \1, \xx\rangle=0$, using \eqref{eq:defn-A-n-0}, we can write
\begin{eq}\label{eq:anx-ineq}
    \|A_n\xx\|_2 &=  \|\big(A_n^0 + \mu_n \1\1^T  - \mu_n I_n \big)\xx\|_2
                 \leq \|A_n^0\xx\|_2 + \mu_n\|\xx\|_2.
\end{eq}
Therefore, we have
\begin{eq}\label{eq:lambda-2-nonzero}
    \Lambda_2 (n) &= \max_{\xx:\langle \1, \xx\rangle=0,  \xx\neq \boldsymbol{0}} \frac{\|A_n\xx\|_2}{\|\xx\|_2}
                  \leq \max_{\xx:\ \xx\neq \boldsymbol{0}} \frac{\|A_n^0\xx\|_2}{\|\xx\|_2} +\mu_n.
\end{eq}
Also, note  that the matrix $H_n = (h_{ij}^n)_{1\leq i,j\leq n}$ defined by $H_n := A_n^0/\sqrt{n}\sigma_n $
satisfies the conditions of \cite[Assumption 2.3]{LS18}, namely
\begin{enumerate}
    \item For all $i\in [n]$, $h^n_{ii} = 0$, and for all $i,j\in[n]$ with $i\neq j$,  $\E[h_{ij}^n] = 0$, $\E[(h_{ij}^n)^2] = \frac{1}{n}$.
    \item Setting $q_n = \sqrt{n}\sigma_n$, by  Assumption~\ref{assumption-1}~\eqref{assumption1-iv}, there exists a fixed constant $c_1 >0$ such that for all $n\geq 1$ and $k\geq 3$,
    \begin{align*}
        \E\big[|h_{ij}^n|^k\big] &\leq \frac{\E\big[|a_{ij}^n - \mu_n|^k\big]}{n^{\frac{k}{2}}\sigma_n^k} 
        \leq \frac{k!}{2}\frac{c^{k-2}\sigma_n^2}{n^{\frac{k}{2}}\sigma_n^k}
        \leq (c_1k)^{c_1k}\frac{1}{nq_n^{k-2}}.
    \end{align*}
Also, \blue{$q_n \gg n^{c_0}$,} due to~Assumption~\ref{assumption-1}~\eqref{assumption1-ii} and further,
    $q_n = O(\sqrt{n})$ since $\sigma_n^2 = O(\mu_n) = O(1)$.
\end{enumerate}
Therefore, by  \cite[Theorem 2.9]{LS18}, for all sufficiently large $n$,
\begin{equation}\label{eq:A-0-evalue}
    \max_{\xx:\ \xx \neq 0} \frac{\|A_n^0\xx\|_2}{\|\xx\|_2} \leq 3\sqrt{n} \sigma_n,
\end{equation}
which then implies \eqref{eq:2nd-evalue-order} using  \eqref{eq:lambda-2-nonzero}.
\end{proof}

Below we state a general version of Lemma~\ref{claim:lambda2-bound} that extends the result to the non-zero diagonal entries case.
\begin{lemma}\label{claim:lambda2-bound-2}
Under {Assumption~\ref{assumption-1}} and the assumptions for non-zero diagonal entries in  Remark~\ref{rem:diag}, the following holds:
\begin{eq}\label{eq:2nd-evalue-order-2}
     \Lambda_2 (n) \leq 3\sqrt{n} \sigma_n +\mu_n + \sqrt{2n(\zeta_n^2+\rho_n^2)}, \quad \PR_0\mbox{-eventually almost surely}. 
\end{eq}
\end{lemma}
The proof of Lemma~\ref{claim:lambda2-bound-2} follows verbatim from the proof of Lemma~\ref{claim:lambda2-bound}, except that the upper bound in~\eqref{eq:lambda-2-nonzero} will be replaced by 
\begin{align*}
    \Lambda_2 (n) = \max_{\xx:\langle \1, \xx\rangle=0,  \xx\neq \boldsymbol{0}} \frac{\|A_n\xx\|_2}{\|\xx\|_2}
    \leq \max_{\xx:\ \xx\neq \boldsymbol{0}} \frac{\|A_n^0\xx\|_2}{\|\xx\|_2} +\mu_n + \Big(\sum_{i=1}^n (a_{ii}^n)^2\Big)^{\frac{1}{2}}.
\end{align*} 
\blue{Using standard concentration bounds \cite[Corollary 2.11]{BLM13} (as used in \eqref{eq:deg-sum}), we can bound
$\sum_{i=1}^n (a_{ii}^n)^2\leq 2n(\zeta_n^2+\rho_n^2)$, $\PR_0$-eventually almost surely. Note that this step requires the moment conditions mentioned in Remark~\ref{rem:diag}.
Rest of the proof is identical to Lemma~\ref{claim:lambda2-bound} since since $A_n^0$ has zero diagonal entries and hence, is omitted.}\\

Next, we prove a bound on the error while approximating $\|A_n\|_{r\to p}$ by $\eta_n(A_n)$.

\begin{lemma}\label{lem:errg-deg}
Under the conditions of Theorem~\ref{thm:asymptotic-normality}, the following holds $\PR_0$-almost surely:  
$$\|A_n\|_{r\to p} =\|A_n\vv_n\|_p= \eta_n (A_n)+o\big( \sigma_n n^{\frac{1}{p} - \frac{1}{r}}\big),$$
where $v_n$ is the maximizer vector in \eqref{def-maxvecn} and $\eta_n(\cdot)$ is defined in \eqref{def-nun}.
\end{lemma}
\begin{proof}
It suffices to show that $\PR_0$-eventually almost surely,
\begin{equation}\label{eq:lem-8.3-bound}
    \big|\|A_n\vv_n \|_{p} - \eta_n(A_n)\big|\leq 
C \frac{\sigma_n^3}{\mu_n^2} n^{\frac{1}{p} - \frac{1}{r}} \sqrt{\frac{\log n}{n}},
\end{equation}
for some constant $C>0$, not depending on $n$.
\blue{Indeed, if~\eqref{eq:lem-8.3-bound} holds, then Lemma~\ref{lem:errg-deg} would follow immediately on observing that
$\sigma_n^2 = O(\mu_n)$ and $\mu_n \gg \sqrt{(\log n)/n}$ by Assumption~\ref{assumption-1}~\eqref{assumption1-ii}.

To show~\eqref{eq:lem-8.3-bound}, note that by Lemma~\ref{lem-verification}, under Assumption~\ref{assumption-1} with associated constants $(\mu_n)_{n \in \N},$ $(\sigma_n)_{n \in \N},$ 
(i) the sequence  $(A_n)_{n \in \N}$  is $\PR_0$-almost surely $(\ve_n, \mu_n)_{n\in \N}$ almost regular in the sense of Definition~\ref{defn:almost-regular} with $\varepsilon_n = \Theta (\sqrt{\frac{\log n}{n\mu_n}\cdot \frac{\sigma_n^2}{\mu_n}})$
and 
(ii) for some constant $c'\in (0,1)$, $(c', \mu_n)_{n\in\N}$ well-connected in the sense of Definition~\ref{defn:two-hop}. 
Also, note that the well-connectedness also implies that $A_n^TA_n$ is irreducible.   
In particular,  the conditions of 
Proposition~\ref{prop:vec-close} are satisfied and  we can 
apply Proposition~\ref{prop:norm-approximation} along with Lemma~\ref{claim:lambda2-bound} to conclude that 
\begin{align*}
\big|\|A_n\vv_n \|_{p} - \eta_n(A_n)\big|
&\leq 
C_1\frac{(3\sqrt{n}\sigma_n + \mu_n)^2}{\mu_n^2n^{3/2 + 1/r}}\sqrt{\frac{\log n}{n\mu_n}\cdot \frac{\sigma_n^2}{\mu_n}} \times  \|A_n\|_{2\to p}\\
&\leq \frac{C_2n\sigma_n^3(1+\frac{\mu_n}{\sqrt{n} \sigma_n})^2}{\mu_n^3n^{3/2+1/r}} \sqrt{\frac{\log n}{n}}\times \|A_n\|_{2\to p}.
\end{align*}

To conclude the proof, we establish the following:

\begin{claim}\label{claim:2-to-p-bound}
 For $p\in [1,2]$, (i) $\|A_n\|_{2\to p} = (1+\oP(1)) \mu_n n^{\frac{1}{2} + \frac{1}{p}}$, (ii) For $p>2$,  $\|A_n\|_{2\to p} \leq C \sqrt{\mu_n} n^{\frac{1}{2} + \frac{1}{p}}$ for some constant $C>0$. 
\end{claim}
\begin{claimproof}
For $1\leq p\leq 2$, the claim is immediate from
Lemma~\ref{lem:ingredients}~\eqref{fact:sigma-order}.
For $p>2$, let $\bld{a}_{i} $ denote the $i$-th row of $A_n$. Then $\|\bld{a}_i \|_2^2 = \sum_{j} a_{ij}^2$ is a sum of independent random variables. Using the Law of Large numbers, $\|\bld{a}_i \|_2^2 \leq C n \sigma^2$ with high probability. 
Therefore,
\begin{eq}
    \|A_n\|_{2\to p} = \max_{\xx: \|\xx\|_2\leq 1} \bigg(\sum_{i\in [n]} |\binner{\bld{a}_i,\xx}|^p\bigg)^{\frac{1}{p}} \leq \max_{\xx: \|\xx\|_2\leq 1}\bigg(\sum_{i\in [n]} \|\bld{a}_i\|_2^p\|\xx\|_2^p\bigg)^{\frac{1}{p}}\leq C (n\mu_n)^{\frac{1}{2}} n^{\frac{1}{p}},
\end{eq}
and this completes the proof of the claim. 
\end{claimproof}}
\noindent
Now observe that
$$\big|\|A_n\vv_n \|_{p} - \eta_n(A_n)\big|
\leq C_1\frac{ \Lambda_2^2(n) \varepsilon_n}{\mu_n n^{1-\frac{1}{p}+\frac{1}{r}}}
\leq C_2 \frac{\sigma_n^3}{\mu_n^2} n^{\frac{1}{p} - \frac{1}{r}} \sqrt{\frac{\log n}{n}},$$
$\PR_0$-eventually almost surely, for constants $C_1,C_2>0$, where
the first inequality is due to Proposition~\ref{prop:norm-approximation} and Claim~\ref{claim:2-to-p-bound}, and 
the last step is due to Lemma~\ref{claim:lambda2-bound} and the choice of~$\varepsilon_n$.
This completes the proof of~\eqref{eq:lem-8.3-bound}.
\end{proof}

\begin{remark}\normalfont
\blue{While we do not believe that the the upper bound on $\|A_n\|_{2\to p}$ given in Claim~\ref{claim:2-to-p-bound} for the hypercontractive case ($p>2$) is tight, it is worthwhile to point out the that the bound $(1+\oP(1))\mu_nn^{\frac{1}{2} + \frac{1}{p}}$ does not work in general if $p>2$.
This can be seen from the following observation:
Recall that $\1$ denotes the $n$-dimensional vector
$(1, 1, \ldots, 1)$ and $e_i$ is the $n$-dimensional vector whose $i$-th component is 1 and all other components are 0.
Then note that for any fixed $i\in [n]$,
$$\frac{\|A_n\1\|_p}{\|\1\|_2} = (1+\oP(1))\mu_nn^{\frac{1}{2} + \frac{1}{p}}\qquad\text{and}\qquad \frac{\|A_n e_i\|_p}{\|e_i\|_2} = (1+\oP(1))(n\mu_n)^{\frac{1}{p}}.$$
Also, 
$$\mu_nn^{\frac{1}{2} + \frac{1}{p}} \ll (n\mu_n)^{\frac{1}{p}}\qquad\text{if and only if} \qquad \mu_n \ll n^{- \frac{p}{2(p-1)}}.$$
Therefore, the vector $e_i$ produces a larger norm value if $\mu_n \ll n^{- \frac{p}{2(p-1)}}$.
As a side-note, this observation hints that if $\mu_n$ scales as $n^{-1/t}$ for some $t>2$, then for all sufficiently large $p$, 
the maximizing vector for $\|A_n\|_{2\to p}$ may not be close to $\1$.}
\end{remark}

\subsection{Proof of asymptotic normality} \label{ssec:asymp-normality}
\blue{We proceed with the proof of asymptotic normality using the Taylor expansion. 
Let $\eta_{n,t}(A_n) := \eta_n(t A_n + (1-t)\E[A_n])$. Thus,  $\eta_{n,1}( A_n) = \eta_n(A_n)$ and $\eta_{n,0}( A_n) = \eta_n(\E[A_n])$.
Using the Taylor expansion of $\eta_{n,t}(A_n)$ with respect to $t$, we obtain 
\begin{eq}\label{eq:taylor}
\eta_{n}( A_n) = \eta_{n}( \E[A_n])
+ \frac{\dif}{\dif t} \eta_{n,t}( A_n)\bigg\vert_{t=0}
+\frac{1}{2} \frac{\dif^2}{\dif t^2} \eta_{n,t}( A_n)\bigg\vert_{t=\xi}
\end{eq}
for some $\xi \in [0,1]$.
The next proposition establishes asymptotics of the above derivative terms.
Recall from~\eqref{eq:eps'-def} that
\begin{equation}\label{eq:eps'-def-2}
   \varepsilon_n'= \frac{\sigma_n^2}{\mu_n^3} \frac{(\log n)^2}{n}.
\end{equation}
\begin{proposition} \label{prop:direct-der}
As $n\to\infty$, 
\begin{eq}
   \frac{\dif}{\dif t} \eta_{n,t}( A_n)\bigg\vert_{t=0} &=  (1+\oP(1))n^{-1+\frac{1}{p} - \frac{1}{r}}\sum_{i,j}(a_{ij}^n - \E[a_{ij}^n])\\
  \frac{\dif^2}{\dif t^2} \eta_{n,t}( A_n) &= (1+\OP(\sqrt{\varepsilon_n'}))\Big[p-1 +\frac{1}{r-1}\Big]\frac{n^{-1+\frac{1}{p} - \frac{1}{r}}}{n\mu_n}\sum_{i=1}^n \Big(\sum_{j=1}^n \big(a_{ij}^n - \E[a_{ij}^n]\big)\Big)^2,
\end{eq}
where $\varepsilon_n'$ is as defined in~\eqref{eq:eps'-def-2} and the $\OP(\sqrt{\varepsilon_n'})$ is uniform over $t\in [0,1]$. 
\end{proposition}}
The proof of Proposition~\ref{prop:direct-der} is deferred to Appendix~\ref{app:proof-8.5}. We now complete the proofs of Theorem~\ref{thm:asymptotic-normality} and 
Theorem~\ref{thm:asymptotic-normality-inhom}. 

\begin{proof}[Proof of Theorem~\ref{thm:asymptotic-normality}] 
Note that Lemma~\ref{lem:errg-deg} ensures that $\eta_n(A_n)$ approximates $\|A_n\|_{r\to p}$ on the fluctuation scale, that is,
$$\big|\|A_n\|_{r\to p} - \eta_n(A_n)\big| = o\big( \sigma_n n^{\frac{1}{p} - \frac{1}{r}}\big)\quad \PR_0\mbox{-almost surely}.$$
Thus, it is enough to prove \eqref{eq:asymp-normal} when  $\|A_n\|_{r\to p}$ is replaced with $\eta_n(A_n)$.  
The first term of the Taylor expansion of $\eta_n(A_n)$ from \eqref{eq:taylor} is
\blue{\begin{equation}\label{eq:expect-term}
    \eta_n(\E[A_n]) = \mu_n(n-1) n^{\frac{1}{p}-\frac{1}{r}}.
\end{equation}}
\blue{Note that $\sum_{i<j} (\aij - \mu_n)$ is a sum of of iid random variables with total variation $s_n^2 := \binom{n}{2} \sigma_n^2$.
By Assumption~\ref{assumption-1}~\eqref{assumption1-iv}, it follows that 
\begin{eq}\label{eq:lypanov}
     \frac{1}{s_n^3} \sum_{i<j} \E [|\aij - \mu_n|^3] \leq C\frac{n^2  \sigma_n^2}{n^3\sigma_n^3} = O\Big(\frac{1}{n\sigma_n}\Big), 
\end{eq}
which is $o(1)$ since $n\sigma_n\to\infty$ by Assumption~\ref{assumption-1}~\eqref{assumption1-ii}. Thus Lyapunov's condition \cite[(27.16)]{Bil95} is satisfied and we can apply the central limit theorem for triangular arrays \cite[Theorem 27.3]{Bil95} to conclude that 
\begin{equation}\label{eq:clt}
 \frac{\sum_{i<j} (\aij - \mu_n)}{s_n}=\frac{\sqrt{2}\sum_{i<j} (\aij - \mu_n)}{\sqrt{n(n-1)} \sigma_n} \dto \mathrm{Normal}(0, 1).  
\end{equation}
Thus, Proposition~\ref{prop:direct-der} shows that the scaled second term on the right hand side of \eqref{eq:taylor} is  
\begin{equation}\label{eq:8.14}
    \frac{1}{\sigma_n n^{\frac{1}{p} - \frac{1}{r}}}\times \frac{\dif}{\dif t} \eta_{n,t}( A_n)\bigg\vert_{t=0} = (1+\oP(1)) \frac{2\sum_{i<j} (\aij - \mu_n)}{n \sigma_n}\dto \mathrm{Normal}(0, 2). 
\end{equation}
To evaluate the third term on the right hand side of \eqref{eq:taylor}, first note that Proposition~\ref{prop:direct-der}, together with~Lemma~\ref{lem:asymp-aux-1}~\eqref{lem:asymp-aux-1-2} implies that for all $\xi \in [0,1]$
\begin{align*}
    \frac{\dif^2}{\dif t^2} \eta_{n,t}( A_n)\bigg\vert_{t=\xi} & = (1+\OP(\sqrt{\varepsilon_n'})) (1+\OP(n^{-1/2})) \Big(p-1+\frac{1}{r-1}\Big) \frac{\sigma_n^2}{\mu_n} n^{\frac{1}{p} - \frac{1}{r}}.
\end{align*}
Now, 
\begin{align*}
    \frac{\sigma_n^2}{\mu_n} n^{\frac{1}{p} - \frac{1}{r}}  \sqrt{\varepsilon_n'} \ll n^{\frac{1}{p} - \frac{1}{r}} \sigma_n \quad \iff
    \quad  \frac{\sigma_n^2}{\mu_n} \Big(\frac{\log^2 n}{n\mu_n^3}\Big)^{1/2} \ll 1
    \impliedby \mu_n \gg \frac{\log^{2/3} n}{n^{1/3}},
\end{align*}
which holds due to Assumption~\ref{assumption-1} \eqref{assumption1-ii}.
Thus, we conclude that 
\begin{eq}\label{eq:8.15}
    \frac{\dif^2}{\dif t^2} \eta_{n,t}( A_n)\bigg\vert_{t=\xi} 
    = \Big(p-1+\frac{1}{r-1}\Big) \frac{\sigma_n^2}{\mu_n} n^{\frac{1}{p} - \frac{1}{r}} +\oP(n^{\frac{1}{p} - \frac{1}{r}} \sigma_n).
\end{eq}
To complete the proof of Theorem~\ref{thm:asymptotic-normality}, substitute 
\eqref{eq:expect-term}, \eqref{eq:8.14}, and \eqref{eq:8.15} into \eqref{eq:taylor}.}
\end{proof}
We now turn to the proof of asymptotic normality in the dense, inhomogeneous case. 
First we will prove a version of Lemma~\ref{claim:lambda2-bound} in this inhomogeneous case.
\begin{lemma}\label{claim:lambda2-bound-inhom}
Let $(A_n)_{n\in \N}$ be a sequence of random matrices satisfying the conditions of {Theorem~\ref{thm:asymptotic-normality-inhom}}.
Then the following holds:
\begin{eq}\label{eq:2nd-evalue-order-inhom}
     \Lambda_2 (n) \leq 3\sqrt{c^*n} \bar{\sigma}_n +\mu_n, \quad \PR_0 \mbox{ eventually almost surely}, 
\end{eq}
where recall that $c^*>0$ is a constant defined in {Assumption~\ref{assumption-inhom}}.
\end{lemma}
As shown below, the proof of this lemma follows on arguments similar to the ones used in Lemma~\ref{claim:lambda2-bound}, with the key difference that 
the bound on the $2\to 2$ norm of the centered random matrix needs a more careful treatment.
\begin{proof}[Proof of Lemma~\ref{claim:lambda2-bound-inhom}]
We first prove the following bound on the centered matrix $A_n^0$ from~\eqref{eq:defn-A-n-0}:
$$\frac{\|A_n^0\xx\|_2}{\|\xx\|_2}\leq 3c^* \sqrt{n}\bar{\sigma}_n, \quad \PR_0 \mbox{ eventually almost surely}.$$
To this end, note that the matrix $H_n = (h_{ij}^n)_{1\leq i,j\leq n}$ defined by $H_n = A_n^0/\sqrt{n}\bar{\sigma}_n $ has the following properties:
\begin{enumerate}
    \item By Assumption~\ref{assumption-inhom}~\eqref{assumption2-i}, $h^n_{ii} = 0$ for all $i\in [n]$ and $\E\big[h_{ij}^n\big] = 0$ for all $i,j\in [n]$, $i\neq j$. 
    \item By Assumption~\ref{assumption-inhom}~\eqref{assumption2-ii}, for all sufficiently large $n$,
    $$\frac{c_*}{n} \leq \min_{i,j}\expt{(h_{ij}^n)^2}\leq \max_{i,j}\expt{(h_{ij}^n)^2} \leq \frac{c^*}{n}$$
    \item Also, by Assumption~\ref{assumption-inhom}~\eqref{assumption2-v}, for all sufficiently large $n$, and every $k\geq 3$
    \begin{align*}
        \E\big[|h_{ij}^n|^k\big] &\leq \frac{\E\big[|a_{ij}^n - \mu_n|^k\big]}{n^{\frac{k}{2}}\bar{\sigma}_n^k} 
        \leq\frac{c_k}{n^{k/2}}.
    \end{align*}
\end{enumerate}
This shows that $H_n$ satisfies the conditions in \cite[Theorem~2.1, Remark~2.2]{AEK19}. Further, by Ger{\v s}gorin's circle theorem \cite{Ger31},   the largest eigenvalue of the matrix $\big(\E[(h_{ij}^n)^2]\big)_{i,j}$ is bounded from above by  $2c^*\bar{\sigma}_n^2$.
An application of  \cite[Theorem~2.1, Remark~2.2]{AEK19}  yields~\eqref{eq:2nd-evalue-order-inhom}. 
\end{proof}

The next lemma proves a version of Lemma~\ref{lem:errg-deg} in the inhomogeneous variance case.
\begin{lemma}\label{lem:errg-deg-inhom}
Let $(A_n)_{n\in\N}$ be a sequence of random matrices satisfying the conditions of {Theorem~\ref{thm:asymptotic-normality-inhom}}. 
Then the following holds $\PR_0$-almost surely: 
$$\|A_n\vv_n \|_{p} = \|A_n\vv_n^{(1)}\|_{p}+o\big( \bar{\sigma}_n n^{\frac{1}{p} - \frac{1}{r}}\big).$$
\end{lemma}
\begin{proof}
The proof follows the proof of Lemma~\ref{lem:errg-deg} verbatim, except that Lemma~\ref{claim:lambda2-bound-inhom} must be used in place of Lemma~\ref{claim:lambda2-bound}.
\end{proof}

\begin{proof}[Proof of Theorem~\ref{thm:asymptotic-normality-inhom}]
Note that Lemma~\ref{lem:errg-deg-inhom} ensures that under the conditions of Theorem~\ref{thm:asymptotic-normality-inhom}, $\eta_n(A_n)$ approximates $\|A_n\|_{r\to p}$ on the fluctuation scale, that is,
$$\big|\|A_n\|_{r\to p} - \eta_n(A_n)\big| = o\big( \bar{\sigma}_n n^{\frac{1}{p} - \frac{1}{r}}\big)\quad \PR_0\mbox{-almost surely}.$$
The rest of proof follows the same steps as the proof of Theorem~\ref{thm:asymptotic-normality}, if one uses 
$\sum_{i<j}\sigma_n^2(i,j)$ in place of $n^2\sigma^2/2$, 
the upper bound $\big(c^*\bar{\sigma}_n)^2$ for the variances of the entries, and the CLT
\begin{equation}\label{eq:clt-inhom}
 \frac{  \sum_{i<j} (\aij - \mu_n)}{ \sqrt{ \sum_{i<j}\sigma_n^2(i,j)}} \dto \mathrm{Normal}(0, 1),  
\end{equation}
in place of \eqref{eq:clt}.
\end{proof}

\section{Relation to the \texorpdfstring{$\ell_r$}{lr} Grothendieck problem}
\label{sec:Groth}
We end this section with the proof of Proposition~\ref{prop:groth}.
\begin{proof}[Proof of  Proposition~\ref{prop:groth}]
Let $x^* \in \R^n$ be a maximizer of $x^T A x$ with $\|x^*\| = 1$. Then, using the method of Lagrange multipliers, there exists $\lag \in \R$ such that if $g: \R^n \mapsto \R$ is the function given by 
\[ g(\xx) = \xx^T A  \xx - \lag \left(\|\xx\|_r^r -1 \right), 
\]
then $\xx^*$ solves the equation 
\begin{equation}
    \label{eq-nabla}
\nabla g(\xx) = 2 A  \xx - \lag r \Psi_r (\xx) = 0,
\end{equation}  
where recall $\Psi_r(\xx) = |\xx|^{r-1} \sgn(\xx).$   
Taking the inner product of $\xx^*$ with the left-hand side of \eqref{eq-nabla} evaluated at $x = \xx^*$, and using the fact that $\langle \xx^*, \Psi_r(\xx^*) \rangle =\|\xx^*\|_r = 1$, it can be seen that 
\begin{equation}\label{eq:groth-lagrange-relation}
    M_r (A ) = \sup_{\|\xx\|_{r}\leq 1} \xx^T A \xx = (\xx^*)^T A  \xx^* = \frac{\lag r}{2}.
\end{equation} 
Now, fix any nonnegative solution $\yy$ of \eqref{eq-nabla}. 
It follows that
\begin{eq}\label{eq:groth-lagrange2}
    \Psi_{r^*}(A^T\yy) = \Big(\frac{\kappa r}{2}\Big)^{\frac{1}{r-1}}\yy
\end{eq}
and also, for $r\geq 2$ and $p = r^*= r/(r-1)$,
\begin{eq}
    &\Psi_p(A \yy) = \Big(\frac{\kappa r}{2}\Big)^{p-1}\Psi_p(\Psi_r(\yy))\\
    \iff &
    A ^T\Psi_p(A \yy) = \Big(\frac{\kappa r}{2}\Big)^{p-1}A^T\Psi_p(\Psi_r(\yy))\\
    \iff & S\yy=
    \Psi_{r^*}(A^T\Psi_p(A\yy)) = \Big(\frac{\kappa r}{2}\Big)^{\frac{p-1}{r-1}}\Psi_{r^*}(A^T\Psi_p(\Psi_r(\yy))).
\end{eq}
Choosing $p = r^*= r/(r-1)$, we have $\Psi_p(\Psi_r(\yy)) = \yy$, and thus
\begin{eq}\label{eq:sy-prop}
   S\yy = \Big(\frac{\kappa r}{2}\Big)^{\frac{p-1}{r-1}}\Psi_{r^*}(A^T\yy)
    = \Big(\frac{\kappa r}{2}\Big)^{\frac{p}{r-1}}\yy,\quad\mbox{due to \eqref{eq:groth-lagrange2}}.
\end{eq}
\blue{Therefore, $S\yy \propto \yy$.
Also, note that since $r\geq 2$, we have 
$p=r^*\leq r$.
Thus, from Lemma~\ref{lem:vcomp}, we know that $S\xx = \newsigma^{\frac{p}{r-1}} \xx$ has a unique solution in $\xx$ that has all positive entries when $A$ is a symmetric matrix with nonnegative entries and $A^TA$ is irreducible (see Proposition~\ref{prop:iteration-converge}).
Since the steps between \eqref{eq-nabla} and~\eqref{eq:sy-prop} consist of implications in both directions, 
we conclude that \eqref{eq-nabla} also has a unique positive solution $\xx^*$} and for $p=r^*$, 
\begin{eq}
    \|A\|_{r\to p}^{\frac{p}{r-1}} = \Big(\frac{\kappa r}{2}\Big)^{\frac{p}{r-1}} 
    &\implies  \|A\|_{r\to p} = \frac{\kappa r}{2}.
\end{eq}
Therefore, \eqref{eq:groth-lagrange-relation} yields that $M_r (A) = \|A\|_{r\to r^*}$ and the proof follows. 
\end{proof}

{\small
\bibliographystyle{plain}
\bibliography{operator-norm-CLT}
}


\appendix
\addtocontents{toc}{\protect\setcounter{tocdepth}{1}}

\blue{
\section{Proof of Proposition~\ref{prop:direct-der}}
\label{app:proof-8.5}
Throughout this appendix, we will omit sub-/superscript $n$. Also, we will repeatedly use the fact that row sums of the $\E[A]$ matrix is $(n-1)\mu = n\mu (1+o(1))$.
Recall
\begin{align*}
    A_t &= tA + (1-t)\expt{A}, \quad\text{for }t\in [0,1],\\
    \bA &= A - \expt{A},\\
    \bd &= \bA\1,\\
    \bm_k &= \binner{\bd^{\oast k}, \1}, \quad k\geq 1.
\end{align*}
Define $ E_t := \Psi_p(A_t\1)$.
We will now calculated the expression of the derivatives, along with the value of the first derivative at $t=0$.
\paragraph{Derivatives of $E_t$.}
Since $E_t= \Psi_p(A_t\1)$,  
\begin{eq}\label{eq:app:der-Et}
E_t' &= (p-1)\Psi_{p-1}(A_t\1) \oast \bd
\\
E_t''& = (p-1)(p-2)\Psi_{p-2}(A_t\1) \oast \bd^{\oast 2}.
\end{eq}
At $t=0$, we have 
\begin{eq}
E_0&= (n\mu)^{p-1}\1 (1+o(1)),
\\E_0' &= (p-1)(n\mu)^{p-2} \bd(1+o(1)).
\end{eq}

\paragraph{Derivatives of $F_t$.}
$F_t = A_t \Psi_p(A_t\1) = A_t E_t$. Then,
\begin{eq}\label{eq:app-ft''}
F_t'&= \bA E_t + A_t E_t',\\
F_t''&= 2\bA E_t' + A_t E_t''.
\end{eq}
At $t=0$, we have 
\begin{eq}
F_0&= (n\mu)^{p}\1(1+o(1)),
\\F_0'&= (n\mu)^{p-2}\big[ n\mu \bd + (p-1)\bm_1 \mu\1\big](1+o(1)).
\end{eq}

\paragraph{Derivatives of $S_t$.}
$S_t  = \Psi_{r'}(F_t)$. Then 
\begin{eq}\label{eq:ders-st}
S_t' &=  (r'-1)\Psi_{r'-1} (F_t)\oast F_t'\\
S_t''&= \Psi_0(F_t)\oast\big[(r'-2)S_t'\oast F_t' + (r'-1) S_t\oast F_t''\big],
\end{eq}
where the second step follows by noting that
\begin{eq}\label{eq:app-f's'}
&F_t\oast S_t' = F_t\oast \big((r'-1)\Psi_{r'-1} (F_t)\oast F_t'\big) = (r'-1)\Psi_{r'} (F_t)\oast F_t' = (r'-1)S_t\oast F_t'\\
&\implies F_t'\oast S_t' + F_t \oast S_t'' = (r'-1)\big[F_t'\oast S_t' + S_t\oast F_t''\big].
\end{eq}
At $t=0$, we have 
\begin{eq}
S_0&=(n\mu)^{\frac{p}{r-1}}\1(1+o(1)),\\
S_0'&= (r'-1)(n\mu)^{\frac{p}{r-1}-2}\big[ n\mu \bd + (p-1)\bm_1 \mu\1\big](1+o(1)).
\end{eq}

\paragraph{Derivatives of $s_t$.}
$s_t = \|S_t\|_r$

\begin{eq}\label{eq:app:sdder-2}
s_t' &= s_t^{-(r-1)}\binner{F_t, S_t'}\\
s_t''&= -(r-1)\frac{(s_t')^2}{s_t}+ s_t^{-(r-1)}\big[\binner{F_t', S_t'} + \binner{F_t, S_t''}\big]\\
&= -(r-1)\frac{(s_t')^2}{s_t}+ (r'-1)s_t^{-(r-1)}
\big[\binner{F_t', S_t'} + \binner{S_t, F_t''}\big]
\end{eq}

\begin{align*}
    & s_t^{r-1}s_t' = \binner{\Psi_r(S_t), S_t'} = \binner{F_t, S_t'}\\
    &\implies (r-1)s_t^{r-2} (s_t')^2 + s_t^{r-1} s_t''=
    \binner{F_t', S_t'} + \binner{F_t, S_t''}\end{align*}
At $t=0$, we have 
\begin{eq}
s_0&= (n\mu)^{\frac{p}{r-1}}n^{\frac{1}{r}}(1+o(1)),\\
s_0'&= p(r'-1)(n\mu)^{\frac{p}{r-1}-1}n^{-1+\frac{1}{r}}\bm_1(1+o(1)).
\end{eq}

\paragraph{Derivatives of $G_t$.}
$G_t  = A_tS_t$.

\begin{eq}\label{eq:app:Gtder}
G_t' &= \bA S_t + A_t S_t'\\
G_t''&= 2\bA S_t' + A_t S_t''.
\end{eq}
At $t=0$, we have 
\begin{eq}
G_0 &= (n\mu)^{\frac{p}{r-1}+ 1}\1(1+o(1))\\
G_0' &=  (n\mu)^{\frac{p}{r-1}-1}\big[ n\mu \bd + p(r'-1)\bm_1 \mu\1\big](1+o(1)).
\end{eq}

\paragraph{Derivatives of $g_t$.}
$g_t  = \|A_tS_t\|_p$.

\begin{eq}\label{eq:ders-gt}
g_t' &= g_t^{-(p-1)}\binner{\Psi_p(G_t), G_t'}, \\
g_t''&= -(p-1)\frac{(g_t')^2}{g_t}+g_t^{-(p-1)}\big[ (p-1)\binner{\Psi_{p-1}(G_t), (G_t')^{\oast 2}} + \binner{\Psi_p(G_t), G_t''}\big],
\end{eq}
where we have used
\begin{align*}
    g_t^{p-1}g_t' &= \binner{\Psi_p(G_t), G_t'},  \\
    (p-1)g_t^{p-2}(g_t')^2 + g_t^{p-1}g_t''
    &= (p-1)\binner{\Psi_{p-1}(G_t), (G_t')^{\oast 2}} + \binner{\Psi_p(G_t), G_t''}.
\end{align*}
At $t=0$, we have 
\begin{eq}
g_0 &=(n\mu)^{\frac{p}{r-1}+1}n^{\frac{1}{p}}(1+o(1)),\\
g_0'&= (n\mu)^{\frac{p}{r-1}}n^{-1+\frac{1}{p}}(p(r'-1)+1)\bm_1(1+o(1)).
\end{eq}
Therefore, at $t=0$,
\begin{eq}\label{eq:app:gtstder}
    \frac{d}{dt}\Big(\frac{g_t}{s_t}\Big)\bigg\vert_{t=0} &= n^{-1+\frac{1}{p} - \frac{1}{r}}\bm_1(1+o(1)).
\end{eq}

\subsection{Auxiliary results}
We start by listing a few auxiliary results that will be used in the calculation of the second derivatives.
Throughout the rest of the appendix, $\varepsilon$ will be  given by \eqref{eq:eps-choice}.
Note that due to Lemma~\ref{lem-verification},
with high probability,
uniformly for all $t \in [0,1]$,
$A_t\1 = n\mu \1 (1+ O(\varepsilon))$, and hence, throughout this section we will use, without reference,
that with high probability, uniformly for all $t \in [0,1]$
\begin{eq}\label{eq:app:etftstgt}
    E_t&= (n\mu)^{p-1}\1 (1+O(\varepsilon))\\
    F_t&= (n\mu)^{p}\1(1+O(\varepsilon))\\
    S_t&=(n\mu)^{\frac{p}{r-1}}\1(1+O(\varepsilon)),\\
    G_t &= (n\mu)^{\frac{p}{r-1}+ 1}\1(1+O(\varepsilon)).
\end{eq}

\begin{lemma} \label{lem:asymp-aux-1}
\label{lem:0.2}
Let $B_{\infty} (\varepsilon):= \{\xx\in \R^n: \|\xx - \1\|_\infty \leq \varepsilon\}$. Then the following hold:
\begin{enumerate}[{\normalfont (i)}]
    \item \label{lem:asymp-aux-1-1} \label{lem:0.2-3} $\|\bd\|_{\infty} \lesssim \varepsilon n\mu$, and $$\sup_{\xx\in B_{\infty} (\varepsilon)}\|\bA\xx - \bd\|_{\infty} \lesssim \varepsilon n\mu.$$

    \item \label{lem:app-lem-item2}
    $|\bm_1| = |\binner{\1, \bd}|\label{lem:asymp-aux-1-2-1}
= \OP( n\sigma \sqrt{\log n})$, and 
    $$\sup_{\xx, \yy \in B_{\infty} (\varepsilon)}|\binner{\xx, \bA \yy} - \binner{\1,\bd}| = \OP( \varepsilon n^{3/2} \sigma).$$
    
    \item \label{lem:asymp-aux-1-2} $\bm_2 = \binner{\1, \bd^{\oast 2}}= n^2\sigma^2 (1+ \OP(n^{-1/2}))$ and  with high probability  
    $$\sup_{\xx, \yy,\zz\in B_{\infty} (\varepsilon)}|\binner{\xx, (\bA\yy)\oast (\bA\zz)} - \binner{\1, \bd^{\oast 2}}|= \OP(  \varepsilon n^2\sigma^2) .$$
    
    \item \label{lem:asymp-aux-1-3}
    $$\sup_{\xx, \yy\in B_{\infty} (\varepsilon)} | \binner{\xx, \bA(\yy\oast \bd)} - \binner{\1, \bd^{\oast 2}} |  =  \OP(\varepsilon n^2 \sigma^2).$$

    \item \label{lem-item:app:1.13}
$\binner{\1, ( A_t \bd )^{\oast 2}} = \OP(n^3 \sigma^2)$,  and uniformly for all $t\in [0,1]$,
\begin{align*}
    \sup_{\xx, \yy\in B_{\infty}(\varepsilon)} \binner{\1, (A_t (\xx \oast (\bA \yy)))^{\oast 2}}   = \OP\big( n^3 \sigma^4 \mu^{-1} (\log n)^2\big). 
\end{align*}
    
\end{enumerate}
\end{lemma}

\begin{proof}
(i)\ The first bound follows from Lemma~\ref{lem-verification}. Also, 
$$\sup_{\xx\in B_{\infty} (\varepsilon)}\|\bA\xx - \bd\|_{\infty}\leq \varepsilon\max_i \sum_{j=1}^n |a_{ij} - \mu| \leq \varepsilon \max_i (d_i+n\mu) \lesssim \varepsilon n\mu. $$

\noindent (ii) The bound on $\binner{\1,\bd}$ follows  using $\var(\binner{\1, \bd}) = O(n^2\sigma^2)$ and Chebyshev's inequality. 
Let $\xx = \1 + \varepsilon \ww_x$ and $\yy = \1+ \varepsilon\ww_y$ with $\|\ww_x\|_{\infty}\leq 1$ and $\|\ww_y\|_{\infty}\leq 1$.
Then
\begin{equation}
    \label{eq:lem-01-1}
     \binner{\xx, \bA \yy} = 
    \binner{\1, \bA\1} + \varepsilon(\binner{\ww_x, \bA \1} + \binner{\1, \bA \ww_y})
    +\varepsilon^2 \binner{\ww_x, \bA \ww_y}
\end{equation}
We have, 
\begin{align*}
    \E\Big[\sum_i |d_i - n\mu|\Big]
    \leq n \sqrt{\E\big[(d_1 - n\mu)^2\big]}=n \sqrt{\sum_j \E(a_{ij} - \mu)^2}
    = O(n^{3/2} \sigma).
\end{align*}
Thus,
\begin{eq}\label{eq:app:2.2}
\sup_{\ww_x\neq \boldsymbol{0}, \|\ww_x\|_{\infty}\leq 1}|\binner{\ww_x, \bA \1}|&=\sum_{i}|d_i - n\mu| =  \OP\big(n^{3/2} \sigma \big),
\\
\sup_{\ww_x,\ww_y\neq \boldsymbol{0}} \Big|\frac{\binner{\ww_x,\bA\ww_y}}{\|\ww_x\|_2 \|\ww_y\|_2}\Big| &\leq \|\bA\|_{2\to 2}= \OP( \sigma\sqrt{n}),  
\end{eq}
where the final step in the second inequality follows using~\eqref{eq:A-0-evalue}.
Also, $\binner{\1, \bA \ww_y} = \binner{\ww_y, \bA\1}$. 
Thus, plugging in the value of $\varepsilon$,  Part (ii) follows from~\eqref{eq:lem-01-1} and~\eqref{eq:app:2.2}.\\

\noindent (iii)
Note that $\binner{\1, \bd^{\oast 2}} = \sum_{i,j,k}  (a_{ij} -\mu)(a_{ik} -\mu),$ and thus,
\begin{align*}
    \E[\binner{\1, \bd^{\oast 2}}]  &= \sum_{i,j,k} \E[(a_{ij} -\mu)(a_{ik} -\mu)]  = (1+O(1/n))n^2 \sigma^2, \\
    \E[\binner{\1, \bd^{\oast 2}}^2]  &= \sum_{\substack{i,j,k \\ i',j',k'}} \E[(a_{ij} -\mu)(a_{ik} -\mu)(a_{i'j'} -\mu)(a_{i'k'} -\mu)]  =   n^4\sigma^4(1+O(1/n)).
\end{align*}
Hence, we can conclude the asymptotics of $\binner{\1, \bd^{\oast 2}}$ using  Chebyshev's inequality. Next, there exists $\ww_x,\ww_y,\ww_z\in \R^n$ such that  $\|\ww_x\|_{\infty}\leq 1$,  $\|\ww_y\|_{\infty}\leq 1$, $\|\ww_y\|_{\infty}\leq 1$, and 
\begin{align} \label{eq:expr-d-in-2}
    &\binner{\xx, (\bA\yy)\oast (\bA\zz)}\\
    &= \binner{\1+ \varepsilon\ww_x, (\bd +\varepsilon\bA\ww_y)\oast (\bd +\varepsilon\bA\ww_z)}\\
    &=\binner{\1, \bd^{\oast 2}}+
    \varepsilon\Big[\binner{\ww_x, \bd^{\oast 2}}+ \binner{\1, \bd\oast (\bA\ww_y)} + \binner{\1, \bd\oast (\bA \ww_z)}\Big]\\
    &\hspace{2cm}+\varepsilon^2\Big[\binner{\ww_x, \bd\oast \bA\ww_y} + \binner{\ww_x, \bd\oast\ww_z}\Big]
   +\varepsilon^3\binner{\ww_x, (\bA\ww_y)\oast (\bA\ww_z)},
\end{align}
where we bound, with high probability,
\begin{align*}
    |\binner{\ww_x, \bd^{\oast 2}}|&\leq \binner{\1, \bd^{\oast 2}}\lesssim n^2\sigma^2,\\
   |\binner{\1, \bd\oast (\bA\ww_y)}| & \leq \|\bd\|_2 \|\bA\|_{2\to 2} \|\ww_y\|_2\lesssim n^2\sigma^2\\
   |\binner{\ww_x, \bd\oast \bA\ww_y}|&\leq \|\bd\oast\ww_x\|_2 \|\bA\|_{2\to 2} \|\ww_y\|_2
   \lesssim n^2\sigma^2\\
   |\binner{\ww_x, (\bA\ww_y)\oast (\bA\ww_z)}|
   &\leq |\binner{\1, |(\bA\ww_y)\oast (\bA\ww_z)|}|
   \leq \|\bA\ww_y\|_2\|\bA\ww_z\|_2\leq \|\bA\|_{2\to 2}^2 n \lesssim n^2\sigma^2.
\end{align*}
Therefore, Part (iii) follows. \\

\noindent
(iv) Note that 
\begin{align*}
    \binner{\xx, \bA(\yy\oast \bd)}
    &= \binner{\1, \bA(\yy\oast \bd)} 
    + \varepsilon \binner{\ww_x, \bA(\yy\oast \bd)}\\
    &= \binner{\1,\bd^{\oast 2}} (1+\varepsilon) +  \varepsilon \binner{\ww_x, \bA(\yy\oast \bd)}.
\end{align*}
Therefore, with high probability, uniformly for all $\xx, \yy \in B_{\infty} (\varepsilon)$,
\begin{align*}
    \big|\binner{\xx, \bA(\yy\oast \bd)} - \binner{\1,\bd^{\oast 2}}\big|&\leq \varepsilon\binner{\1,\bd^{\oast 2}}+ \varepsilon \|\ww_x\|_2 \|\bA\|_{2\to 2} \|\yy\oast \bd\|_2\lesssim \varepsilon n^2 \sigma^2 = \OP(\varepsilon\bm_2),
\end{align*}
where we have again used that $\|\bA\|_{2\to 2} \lesssim \sigma \sqrt{n}.$\\

\noindent (v) Note that 
\begin{align*}
\E\big[\binner{\1, ( A_t \bd )^{\oast 2}}  \big]    = \sum_i \sum_{j,k,j',k'}\E\big[ a^t_{ij} (a_{jk} - \mu)a^t_{ij'} (a_{j'k'} - \mu) \big].
\end{align*}
We can only have a non-zero contribution from an expectation term only if  $\{j,k\}$ equals one of $\{i,j\}, \{i,j'\}, \{j',k'\}$, and, $\{j',k'\}$ equals one of $\{i,j\}, \{i,j'\}, \{j,k\}$. 
This implies that $i = k = k'$ or $\{j, k\} = \{j', k'\}$.
In both cases, there are at most $n^3$ choices of the indices, and each of the terms can be at most $O(\sigma^2)$ (using Assumption~\ref{assumption-1} \eqref{assumption1-iv} to bound the higher moments). 
Therefore, applying Markov's inequality yields
\begin{eq}\label{eq:app:1.13}
\binner{\1, ( A_t \bd )^{\oast 2}} = \OP(n^3\sigma^2).
\end{eq}
Next, 
\begin{eq}\label{app:eq: 1.08-b}
    A_t (\xx \oast (\bA \yy)) =  
    A_t \bd 
    +\varepsilon A_t (\ww_x \oast \bd)
    +\varepsilon A_t (\bA \ww_y)
    +\varepsilon^2 A_t (\ww_x \oast (\bA \ww_y)).
\end{eq}
Thus, 
\begin{eq}\label{eq:app:1.14}
\binner{\1,(\varepsilon A_t (\ww_x \oast \bd))^{\oast 2}} 
\leq  \varepsilon^2\| A_t  |\bd|\|_{2}^2\lesssim \varepsilon^2 (n\mu)^2  \|\bd\|_2^2 =  \OP(  n^3  \sigma^4 \log n). 
\end{eq}
Also, 
\begin{align*}
   \big|( \varepsilon A_t (\bA \ww_y))_i\big|
    &= \varepsilon \big| \sum_{j,k}a_{ij}^t \ba_{jk} (\ww_y)_k \big|
    =\varepsilon \big|\sum_k (\ww_y)_k \sum_{j}a_{ij}^t \ba_{jk}  \big|
    \lesssim \varepsilon \sum_{k} \big|\sum_j a_{ij}^t \ba_{jk} \big|,
\end{align*}
and thus, 
\begin{eq}\label{eq:app:eq-1.13-a}
    \binner{\1,(\varepsilon A_t (\bA \ww_y))^{\oast 2}} 
    &\leq \varepsilon^2 \sum_i \Big( \sum_{k} \big|\sum_j a_{ij}^t \ba_{jk} \big|\Big)^2.
\end{eq}
Taking expectation, 
\begin{eq}\label{eq:app:eq-1.13}
\sum_i \E\Big( \sum_{k} \big|\sum_j a_{ij}^t \ba_{jk} \big|\Big)^2
\leq 
\sum_i \Big( \sum_{k} \Big[ \sum_{j, j'} \E\big(a_{ij}^t \ba_{jk} a_{ij'}^t \ba_{j'k}\big)\Big]^{1/2} \Big)^2,
\end{eq}
where we have used the following fact:
\begin{fact}
For any collection of real-valued random variables $\{X_1, X_2, \ldots, X_n\}$,
$$\E\Big(\sum_{k}|X_k|\Big)^2 \leq \Big(\sum_k\big(\E[X_k^2]\big)^{1/2}\Big)^2.$$
\end{fact}
\noindent
Indeed, the above fact can be seen by using the Cauchy-Schwarz inequality.
Now, the expectation terms in~\eqref{eq:app:eq-1.13} can be non-zero only if $j=j'$ or $k=i$. Thus, for any fixed $i$, when $k=i$, we have
\begin{align*}
    \Big[ \sum_{j, j'} \E\big(a_{ij}^t \ba_{ji} a_{ij'}^t \ba_{j'i}\big)\Big]^{1/2} = O(n(\mu \sigma^2)^{1/2}),
\end{align*}
and, when $k\neq i$,
\begin{align*}
 \Big[ \sum_{j, j': j = j'} \E\big(a_{ij}^t \ba_{jk} a_{ij'}^t \ba_{j'k}\big)\Big]^{1/2}  = O((n\mu \sigma^2)^{1/2})
\end{align*}
Therefore, plugging the bounds in~\eqref{eq:app:eq-1.13}, we get 
$$\sum_i \E\Big( \sum_{k} \big|\sum_j a_{ij}^t \ba_{jk} \big|\Big)^2 = O(n^{4}\mu\sigma^2),$$
and hence, from~\eqref{eq:app:eq-1.13-a},
\begin{eq}\label{eq:app:1.17}
\binner{\1,(\varepsilon A_t (\bA \ww_y))^{\oast 2}} = \OP\big(n^3\sigma^4\mu^{-1}\log n\big).
\end{eq}
Next, 
\begin{eq}\label{eq:app:1.18}
   \binner{\1,(\varepsilon^2 A_t (\ww_x \oast (\bA \ww_y)))^{\oast 2}} 
   \leq \varepsilon^4  \binner{\1,( A_t \big(|\bA|\1\big) )^{\oast 2}}
   = \OP(\varepsilon^4 n^5\mu^2\sigma^2) = \OP(n^3\sigma^4\mu^{-1}(\log n)^2),
\end{eq}
where $|\bA| = (|a_{ij} - \mu|)_{i,j}$.
Therefore, using~\eqref{eq:app:1.13}, \eqref{eq:app:1.14}, \eqref{eq:app:1.17},~\eqref{eq:app:1.18}, and the fact that for any $x_i \in \R$, $i=1,2,3,4$, $(x_1 + x_2 + x_3 + x_4)^4 \leq 16(x_1^4 + x_2^4 + x_3^4 + x_4^4)$, we get
\begin{eq}\label{eq:app-1.20}
 \sup_{\xx, \yy \in B_\infty(\varepsilon)} \big|\binner{\1, (A_t (\xx \oast (\bA \yy)))^{\oast 2}}\big| 
    &= \OP\big(n^3 \sigma^4 \mu^{-1} (\log n)^2\big),
\end{eq} 
and the proof follows.
\end{proof}

\subsection{Calculation of second derivatives at arbitrary point}

Our goal is to calculate $\frac{d^2}{dt^2}\Big(\frac{g_t}{s_t}\Big)$ at an arbitrary point $t \in [0,1]$.

\subsubsection{Derivative of $s_t$ as given in~\eqref{eq:app:sdder-2}}
\label{app:ssec-sders}
The goal of this section is to prove the following lemma:
\begin{lemma}\label{lem:app:b.3}
Uniformly over $t\in [0,1]$, 
\begin{align*}
    |s_t'|&\lesssim (n\mu)^{\frac{p}{r-1}- 1}n^{\frac{1}{r}}\sqrt{\log n}\cdot\frac{\sigma^2}{\mu}\\
    s_t'' &= (r'-1)(r'-1+p(p-1))(n\mu)^{\frac{p}{r-1}-2}n^{-1+\frac{1}{r}}\bm_2 (1+\OP(\sqrt{\varepsilon'})).
\end{align*}
\end{lemma}

To prove Lemma~\ref{lem:app:b.3}, we need to calculate mainly three terms: $\binner{F_t, S_t'}$, $\binner{F_t', S_t'}$, and 
$\binner{S_t, F_t''}$.
We will calculate the values of these terms in this section at an arbitrary point $t\in [0,1]$.
Let us denote by $\xx,\yy,\zz$ etc generic variable vectors in $B_{\infty} (\varepsilon):= \{\xx\in \R^n: \|\xx - \1\|_\infty \leq \varepsilon\}$, which can change values from line to line.

\paragraph{Calculating $\binner{F_t, S_t'}$.}
From~\eqref{eq:app-f's'}, note that 
\begin{align*}
    \big|\binner{F_t, S_t'} \big| 
    &= (r'-1) \big|\binner{S_t, F_t'} \big|
     = (r'-1) \big|\binner{S_t, \bA E_t + A_t E_t'} \big|\\
    &\lesssim (n\mu)^{\frac{p}{r-1}+ p- 2}\big[n\mu\binner{\xx, \bA\yy} + \binner{\xx, \zz\oast \bd}]\lesssim  (n\mu)^{\frac{p}{r-1}+ p- 1}\varepsilon n^{3/2}\sigma,
\end{align*}
where in the last step, we have used Lemma~\eqref{lem:asymp-aux-1-2-1} and the fact that $\binner{\xx, \zz\oast \bd} = \binner{\xx\oast\zz,  \bA\1}$. 

\paragraph{Calculating $\binner{F_t', S_t'}$.}
Due to \eqref{eq:ders-st},
\begin{eq}\label{eq:app:local-2.12}
    \binner{F_t', S_t'}
    &= (r'-1)\binner{F_t', \Psi_{r'-1}(F_t)\oast F_t'}\\
    &=(r'-1)\binner{\Psi_{r'-1}(F_t), (\bA E_t)^{\oast 2} + (A_t E_t')^{\oast 2} +2(\bA E_t)\oast (A_t E_t')}
\end{eq}
Using Lemma~\ref{lem:asymp-aux-1}~\eqref{lem:asymp-aux-1-2}, 
\begin{eq}\label{eq:0.14-a}
\binner{\Psi_{r'-1}(F_t), (\bA E_t)^{\oast 2}} &= (n\mu)^{\frac{p}{r-1}+p-2} \Big[\binner{\1, \bd^{\oast 2}}+\OP(\varepsilon n^2\sigma^2)\Big] \\
& = (1+\OP(\varepsilon))(n\mu)^{\frac{p}{r-1}+p-2} \bm_2.
\end{eq}
Next, due to Lemma~\ref{lem:asymp-aux-1}~\eqref{lem:asymp-aux-1-2} and~\eqref{lem-item:app:1.13}, uniformly for any $\xx\in B_{\infty}(\varepsilon)$,
\begin{eq}
\binner{\1, (A_t(\xx\oast \bd))^{\oast 2}} &
=\OP\big( n(\log n)^2\frac{\sigma^2}{\mu} \bm_2 \big).
\end{eq}
Therefore, 
\begin{eq}
|\binner{\Psi_{r'-1}(F_t), (A_t E_t')^{\oast 2}}| 
= \OP\big((n\mu)^{\frac{p}{r-1}+p-2} \bm_2 \varepsilon'\big),
\end{eq}
where $\varepsilon'= \frac{\sigma^2}{\mu^3} \frac{(\log n)^2}{n}$ is as defined in~\eqref{eq:eps'-def-2}.
Finally, 
\begin{eq}\label{eq:app:1.4}
|\binner{\Psi_{r'-1}(F_t),(\bA E_t)\oast (A_t E_t')}| &\leq \max_i \big(\Psi_{r'-1}(F_t)\big)_i \times \binner{\1,|(\bA E_t)\oast (A_t E_t')|}\\
& \lesssim (n\mu)^{\frac{p}{r-1}-p}  \binner{\1,|(\bA E_t)\oast (A_t E_t')|} \\ 
& \leq (n\mu)^{\frac{p}{r-1}-p}  \binner{\1, (\bA E_t)^{\oast 2}}^{1/2} \binner{\1, (A_t E_t')^{\oast 2}}^{1/2} \\
& = \OP((n\mu)^{\frac{p}{r-1}+p-2} \bm_2\sqrt{\varepsilon'}). 
\end{eq}
Therefore, plugging the estimates in~\eqref{eq:app:local-2.12},
\begin{eq}\label{eq:app:ft'st'}
 \binner{F_t', S_t'} = (1+\OP(\sqrt{\varepsilon'}))(r'-1)(n\mu)^{\frac{p}{r-1}+p-2} \bm_2, 
\end{eq}
where we have used the fact that $\sqrt{\varepsilon'}\gg \varepsilon.$

\paragraph{Calculating $\binner{S_t,F''_t}$.}
Note, using \eqref{eq:app-ft''}, we get that
\begin{eq}
\binner{S_t,F''_t}  & = \binner{S_t,2\bA E_t' + A_t E_t''}.
\end{eq}
Now, due to~\eqref{eq:app:der-Et}, and Lemma~\ref{lem:asymp-aux-1}~\eqref{lem:asymp-aux-1-2} and~\eqref{lem:asymp-aux-1-3},
\begin{align*}
   \binner{S_t,\bA E_t'}
   &= (p-1)\binner{S_t, \bA(\Psi_{p-1}(A_t\1)\oast\bd)} = (p-1) (n\mu)^{\frac{p}{r-1} +p-2} \bm_2 (1+\OP(\varepsilon)),
\end{align*}
and 
\begin{align*}
\binner{S_t,A_t E_t''}  & = (p-1)(p-2) (n\mu)^{\frac{p}{r-1} +p-2}\bm_2 (1+\OP(\varepsilon)).
\end{align*}

\begin{proof}[Proof of Lemma~\ref{lem:app:b.3}]
Using~\eqref{eq:app:sdder-2} and the estimates derived in this section, we get that, uniformly over $t\in [0,1]$, 
\begin{align*}
    |s_t'|&\lesssim (n\mu)^{\frac{p}{r-1}- 1}n^{\frac{1}{r}}\sqrt{\log n}\cdot\frac{\sigma^2}{\mu}\\
    s_t'' &= (r'-1)(r'-1+p(p-1))(n\mu)^{\frac{p}{r-1}-2}n^{-1+\frac{1}{r}}\bm_2 (1+\OP(\sqrt{\varepsilon'})). 
\end{align*}
\end{proof}

\subsubsection{Derivative of $g_t$ as given in~\eqref{eq:ders-gt}}
The goal of this section is to prove the following lemma:
\begin{lemma}\label{lem:app:b.4}
Uniformly over $t\in [0,1]$, 
\begin{align*}
    |g_t'| &\lesssim  (n\mu)^{\frac{p}{r-1}} n^{\frac{1}{p}}\sqrt{\log n}\cdot\frac{\sigma^2}{\mu}\\
    g_t'' 
&=\Big[p-1 +(r'-1)\Big(p(p-1) +\frac{1}{r-1} +1\Big)\Big] (n\mu)^{\frac{p}{r-1}-1}n^{-1+\frac{1}{p}}\bm_2(1+ \OP(\sqrt{\varepsilon'})).
\end{align*}
\end{lemma}

Similar to Section~\ref{app:ssec-sders}, the proof of Lemma~\ref{lem:app:b.4} requires three terms: $\binner{\Psi_p(G_t), G_t'}$, $\binner{\Psi_{p-1}(G_t), (G_t')^{\oast 2}}$,
$\binner{\Psi_p(G_t), G_t''}$.
We will calculate the values of these terms in this section at an arbitrary point $t\in [0,1]$.
Recall~\eqref{eq:app:etftstgt}.

\paragraph{Calculating $\binner{\Psi_p(G_t), G_t'}$.}
\begin{align*}
    \binner{\Psi_p(G_t), G_t'}
    &=  \binner{\Psi_p(G_t), \bA S_t + A_t S_t'}\\
    &= \binner{\Psi_p(G_t), \bA S_t} + (r'-1)\binner{\Psi_p(G_t), A_t \big(\Psi_{r'-1} (F_t)\oast F_t'\big)}\\
    &=\binner{\Psi_p(G_t), \bA S_t} + (r'-1)\binner{\Psi_p(G_t), A_t \big(\Psi_{r'-1} (F_t)\oast (\bA E_t + A_t E_t')\big)}
\end{align*}
Therefore, from Lemma~\ref{lem:asymp-aux-1}~\eqref{lem:app-lem-item2},
\begin{align*}
    \binner{\Psi_p(G_t), \bA S_t} = (n\mu)^{\frac{p(p-1)}{r-1} + p-1 +\frac{p}{r-1}}\binner{\xx, \bA\yy}
    \lesssim (n\mu)^{\frac{p^2}{r-1} + p-1} \varepsilon n^{3/2}\sigma.
\end{align*}
Also, 
\begin{eq}
&\big|\binner{\Psi_p(G_t), A_t \big(\Psi_{r'-1} (F_t)\oast (\bA E_t + A_t E_t')\big)}\big|\\
&=\big|\binner{\Psi_{r'-1} (F_t)\oast(A_t\Psi_p(G_t)),   \bA E_t }\big| + \big|\binner{\Psi_{r'-1} (F_t)\oast(A_t\Psi_p(G_t)),   A_t E_t'}\big|\\
&\lesssim (n\mu)^{\frac{p(p-1)}{r-1} + p-1 +\frac{p}{r-1}}\big|\binner{\xx, \bA\yy}\big| +
\big|\binner{\big(A_t(\Psi_{r'-1} (F_t)\oast(A_t\Psi_p(G_t)))\big)\oast\Psi_{p-1}(A_t\1), \bA\1}\big|\\
&\lesssim (n\mu)^{\frac{p^2}{r-1} + p-1} \varepsilon n^{3/2}\sigma,
\end{eq}
where the last inequality uses Lemma~\ref{lem:asymp-aux-1}~\eqref{lem:app-lem-item2} again.

\paragraph{Calculating $\binner{\Psi_{p-1}(G_t), (G_t')^{\oast 2}}$.}
First, due to~\eqref{eq:app:Gtder},
\begin{eq}\label{eq:app2-1.8}
    \binner{\Psi_{p-1}(G_t), (G_t')^{\oast 2}}
    = \binner{\Psi_{p-1}(G_t), (\bA S_t)^{\oast 2} + (A_t S_t')^{\oast 2} + 2(\bA S_t)\oast (A_t S_t') }
\end{eq}
Similarly to \eqref{eq:0.14-a}, Lemma~\ref{lem:asymp-aux-1}~\eqref{lem:asymp-aux-1-2} yields
\begin{align}\label{eq:app2-1.9}
     \binner{\Psi_{p-1}(G_t), (\bA S_t)^{\oast 2}}
     = 
    (n\mu)^{\frac{p^2}{r-1} + p-2}\bm_2 (1+ \OP(\varepsilon)).
\end{align}
Now,
\begin{eq}\label{eq:0.14-a-2}
    &\binner{\Psi_{p-1}(G_t), (A_t S_t')^{\oast 2} }\\
    &\lesssim  \binner{\Psi_{p-1}(G_t), (A_t (\Psi_{r'-1} (F_t)\oast (\bA E_t + A_t E_t')))^{\oast 2} }\\
    &\lesssim 2\binner{\Psi_{p-1}(G_t), (A_t (\Psi_{r'-1} (F_t)\oast (\bA E_t)))^{\oast 2} +(A_t (\Psi_{r'-1} (F_t)\oast (A_t E_t')))^{\oast 2} }\\
    &\lesssim (n\mu)^{\frac{p(p-2)}{r-1} +p-2} \binner{\1, (A_t (\Psi_{r'-1} (F_t)\oast (\bA E_t)))^{\oast 2} +(A_t (\Psi_{r'-1} (F_t)\oast (A_t E_t')))^{\oast 2} },
\end{eq}
where the last inequality uses~\eqref{eq:app:etftstgt} and the fact that each term of $(A_t (\Psi_{r'-1} (F_t)\oast (\bA E_t)))^{\oast 2} $ and $(A_t (\Psi_{r'-1} (F_t)\oast (A_t E_t')))^{\oast 2}$ is nonnegative.
We will calculate the two terms in~\eqref{eq:0.14-a-2} separately. 
For the first term, we can write
\begin{eq}\label{app:eq: 1.08-a}
    \big|\binner{\1, (A_t (\Psi_{r'-1} (F_t)\oast (\bA E_t)))^{\oast 2} }\big|
    &=(n\mu)^{\frac{2p}{r-1} -2}\big|\binner{\1, (A_t (\xx \oast (\bA \yy)))^{\oast 2}}\big|\\
    &=  \OP\big((n\mu)^{\frac{2p}{r-1}}\varepsilon'\bm_2\big),
\end{eq}
where $\varepsilon'$ is defined in~\eqref{eq:eps'-def-2} and
the last equality uses Lemma~\ref{lem:asymp-aux-1}~\eqref{lem-item:app:1.13}. 

Next, using~\eqref{eq:app:der-Et} for the second term in \eqref{eq:0.14-a-2},
\begin{eq}\label{eq:app:1.21-1}
    |\binner{\1, (A_t (\Psi_{r'-1} (F_t)\oast (A_t E_t')))^{\oast 2} }|
    &\lesssim (n\mu)^{\frac{2p}{r-1} - 4} \sup_{\xx, \yy\in B_{\infty}(\varepsilon)} |\binner{\1, (A_t (\xx\oast (A_t (\yy \oast \bd))))^{\oast 2} }|\\
    &\lesssim (n\mu)^{\frac{2p}{r-1} - 4}\Big[
    |\binner{\1, \big(A_t^2\bd\big)^{\oast 2}}| + \varepsilon^2 \binner{\1, (A_t^2  |\bd|)^{\oast 2} }\Big].
\end{eq}
Now,
\begin{align*}
     \E|\binner{\1, \big(A_t^2\bd\big)^{\oast 2}}| 
     &= \sum_i \E\Big( \sum_{j, k, l} a_{ij}a_{jk} (a_{kl} - \mu)\Big)^2 \\
     &=\sum_i  \sum_{\substack{j, k, l \\ j',k',l'}} \E[a_{ij}a_{jk} (a_{kl} - \mu) a_{ij'}a_{j'k'} (a_{k'l'} - \mu)] = O(n^5\max(\mu^4\sigma^2, \mu^2\sigma^4)), 
\end{align*}
where, in the above sum, the expectation will be non-zero only if $\{k, l\}$ is same as one of $\{i,j\}, \{j,k\}, \{i,j'\}, \{j'k'\}, \{k',l'\}$, and, 
$\{k', l'\}$ is same as one of $\{i,j\}, \{j,k\}, \{i,j'\}, \{k, l\}, \{j', k'\}$.
There are at most $n^5$ such choices of indices and the main contribution comes from the case when there are 5 distinct indices.
In that case, each term is at most $O(\max(\mu^4\sigma^2, \mu^2\sigma^4))$.
Also, for the second term in~\eqref{eq:app:1.21-1}, using Lemma~\ref{lem:0.2}~\eqref{lem:0.2-3}, we get
\begin{align*}
    \E\Big[\varepsilon^2 \binner{\1, (A_t^2  |\bd|)^{\oast 2} }\Big] \lesssim n^5 \mu^2\sigma^4 \log^2 n. 
\end{align*}
Therefore, from~\eqref{eq:app:1.21-1}, we get
\begin{eq}\label{eq:app:2.27}
    |\binner{\1, (A_t (\Psi_{r'-1} (F_t)\oast (A_t E_t')))^{\oast 2} }|
   &= \OP\Big((n\mu)^{\frac{2p}{r-1}}\max \big\{n \sigma^2 , n\frac{\sigma^4}{ \mu^{2}}\log^2 n\big\}\Big)\\
   &= \OP\Big((n\mu)^{\frac{2p}{r-1}} \bm_2\max \big\{\frac{1}{n}, \frac{\sigma^2}{ n\mu^{2}}\log^2 n\big\}\Big) \\
   & = \OP((n\mu)^{\frac{2p}{r-1}} \bm_2 \varepsilon' ),
\end{eq}
where $\varepsilon'$ is given by \eqref{eq:eps'-def-2}.
Thus, plugging in the estimates from~\eqref{app:eq: 1.08-a} and~\eqref{eq:app:2.27} into~\eqref{eq:0.14-a-2}, we get
\begin{eq}\label{eq:app:1.22}
     |\binner{\Psi_{p-1}(G_t), (A_t S_t')^{\oast 2} }| = \OP\big((n\mu)^{\frac{p^2}{r-1} + p-2}\bm_2 \varepsilon'\big).
\end{eq}
Finally, similar to~\eqref{eq:app:1.4}, using~\eqref{eq:app2-1.9} and~\eqref{eq:app:1.22}, we can write that
\begin{eq}\label{eq:app:1.29}
\binner{\Psi_{p-1}(G_t), (\bA S_t)\oast (A_t S_t') } = \OP\big((n\mu)^{\frac{p^2}{r-1} + p-2}\bm_2\sqrt{\varepsilon'}\big).
\end{eq}
Therefore, using \eqref{eq:app2-1.9}, \eqref{eq:app:1.22}, and~\eqref{eq:app:1.29}, we get that uniformly over $t\in [0,1]$,
\begin{eq}
\binner{\Psi_{p-1}(G_t), (G_t')^{\oast 2}} = 
(n\mu)^{\frac{p^2}{r-1} + p-2}\bm_2 (1+ \OP(\sqrt{\varepsilon'})).
\end{eq}

\paragraph{Calculating $\binner{\Psi_p(G_t), G_t''}$.}
Using~\eqref{eq:app:Gtder},
\begin{eq}\label{eq:app:2.31}
\binner{\Psi_p(G_t), G_t''}
& = \binner{\Psi_p(G_t), 2\bA S_t' + A_t S_t''}\\
&=2(r'-1)\binner{\Psi_p(G_t), \bA (\Psi_{r'-1} (F_t)\oast F_t')} \\
&\hspace{1cm}
+ \binner{\Psi_p(G_t), A_t \big(\Psi_0(F_t)\oast\big[(r'-2)S_t'\oast F_t' + (r'-1) S_t\oast F_t''\big]\big)}
\end{eq}
As before, we will calculate the above terms separately.
\begin{eq}\label{eq:app:2.32}
    &\binner{\Psi_p(G_t), \bA (\Psi_{r'-1} (F_t)\oast F_t')}\\
    &= \binner{\Psi_p(G_t), \bA (\Psi_{r'-1} (F_t)\oast (\bA E_t + A_t E_t'))}\\
    &= \binner{\Psi_p(G_t), \bA (\Psi_{r'-1} (F_t)\oast (\bA E_t + (p-1) A_t(\Psi_{p-1}(A_t\1) \oast \bd)))}\\
    &= \binner{\Psi_p(G_t), \bA (\Psi_{r'-1} (F_t)\oast (\bA E_t ))}
    +(p-1)\binner{\Psi_p(G_t), \bA (\Psi_{r'-1} (F_t)\oast ( A_t(\Psi_{p-1}(A_t\1) \oast \bd)))}.
\end{eq}
For the first term in~\eqref{eq:app:2.32}, due to Lemma~\ref{lem:asymp-aux-1}~\eqref{lem:asymp-aux-1-2}
\begin{eq}
    \binner{\Psi_p(G_t), \bA (\Psi_{r'-1} (F_t)\oast (\bA E_t ))}
    &=(n\mu)^{\frac{p^2}{r-1} + p -2} \binner{\xx, \bA(\yy \oast (\bA\zz))}\\
    &= (n\mu)^{\frac{p^2}{r-1} + p -2}\bm_2 (1 + \OP(\varepsilon)).
\end{eq}
For the second term in~\eqref{eq:app:2.32}, 
\begin{eq}
    &|\binner{\Psi_p(G_t), \bA (\Psi_{r'-1} (F_t)\oast ( A_t(\Psi_{p-1}(A_t\1) \oast \bd)))}|\\
    &= (n\mu)^{\frac{p^2}{r-1} + p -3}|\binner{\yy\oast (\bA\xx), A_t(\zz \oast \bd)}|\\
    &\leq (n\mu)^{\frac{p^2}{r-1} + p -3} \big[\binner{\1, \big(\yy\oast (\bA\xx)\big)^{\oast 2}}\big]^{1/2} \big[\binner{\1, (A_t(\zz\oast\bd))^{\oast 2}}\big]^{1/2}\\
    &\leq (n\mu)^{\frac{p^2}{r-1} + p -3}
    \big[\binner{\yy^{\oast 2},  (\bA\xx)^{\oast 2}}\big]^{1/2}\times \big[\binner{\1, (A_t(\zz\oast\bd))^{\oast 2}}\big]^{1/2}\\
    &\lesssim \OP((n\mu)^{\frac{p^2}{r-1} + p -2}\bm_2\sqrt{\varepsilon'}),
\end{eq}
where in the last inequality, we have used 
Lemma~\ref{lem:asymp-aux-1}~\eqref{lem:asymp-aux-1-2} and~\eqref{eq:app-1.20}. 
Therefore,~\eqref{eq:app:2.32} yields
\begin{eq}\label{eq:app:1:30}
    \binner{\Psi_p(G_t), \bA (\Psi_{r'-1} (F_t)\oast F_t')} = (n\mu)^{\frac{p^2}{r-1} + p -2}\bm_2 (1 + \OP(\sqrt{\varepsilon'})).
\end{eq}
Next, from~\eqref{eq:ders-st}, note that 
\begin{eq}
     S_t'\oast F_t' 
    =  (r'-1)\Psi_{r'-1} (F_t)\oast (F_t')^{\oast 2},
\end{eq}
and thus, each term in $S_t'\oast F_t'$ is nonnegative, $\PR_0$-almost surely.
Therefore, we can write using~\eqref{eq:app:ft'st'},
\begin{eq}\label{eq:app:1.27-1}
    \binner{\Psi_p(G_t), A_t \big(\Psi_0(F_t)\oast S_t'\oast F_t'\big)} 
    &=  \binner{(A_t\Psi_p(G_t))\oast \Psi_0(F_t),  S_t'\oast F_t'} \\
    &= (n\mu)^{\frac{p(p-1)}{r-1}}\binner{\1, S_t'\oast F_t'}(1 + \OP(\varepsilon)) \\
    &= (r'-1)(n\mu)^{\frac{p^2}{r-1} + p -2}\bm_2 (1 + \OP(\sqrt{\varepsilon'})).
\end{eq}
Also, from~\eqref{eq:app-ft''} we get
\begin{eq}\label{eq:app:local-1.26}
    &\binner{\Psi_p(G_t), A_t \big(\Psi_0(F_t)\oast S_t\oast F_t''\big)}\\
    &= 2\binner{\Psi_p(G_t), A_t \big(\Psi_0(F_t)\oast S_t\oast [\bA E_t']\big)}
    +\binner{\Psi_p(G_t), A_t \big(\Psi_0(F_t)\oast S_t\oast [A_t E_t'']\big)}\\
    &= 2\binner{\Psi_p(G_t), A_t \big(\Psi_0(F_t)\oast S_t\oast [\bA E_t']\big)}\\
    &\hspace{1cm}+(p-1)(p-2)\binner{\Psi_p(G_t), A_t \big(\Psi_0(F_t)\oast S_t\oast [A_t (\Psi_{p-2}(A_t\1) \oast \bd^{\oast 2})]\big)}\\
    &=p(p-1) (n\mu)^{\frac{p^2}{r-1}+p-2}\bm_2(1+\OP(\varepsilon)),
\end{eq}
where in the last step, we have used Lemma~\ref{lem:asymp-aux-1}~\eqref{lem:asymp-aux-1-3} and that
\begin{align*}
    &\binner{\Psi_p(G_t), A_t \big(\Psi_0(F_t)\oast S_t\oast [\bA E_t']\big)} \\
    &= 
(p-1)\binner{(A_t\Psi_p(G_t))\oast\Psi_0(F_t)\oast S_t,  \bA (\Psi_{p-1}(A_t\1) \oast \bd)}\\
&= (p-1) (n\mu)^{\frac{p^2}{r-1}+p-2}\bm_2(1+\OP(\varepsilon)).
\end{align*}
Plugging in the values from~\eqref{eq:app:1:30},~\eqref{eq:app:1.27-1}, and~\eqref{eq:app:local-1.26} into~\eqref{eq:app:2.31},
\begin{eq}
\binner{\Psi_p(G_t), G_t''} 
&= 
(n\mu)^{\frac{p^2}{r-1} + p -2}\bm_2\Big[2(r'-1) + (r'-2)(r'-1) + (r'-1)p(p-1)\Big] (1 + \OP(\sqrt{\varepsilon'}))\\
&=(r'-1)\Big(p(p-1) +\frac{1}{r-1} +1\Big)(n\mu)^{\frac{p^2}{r-1} + p -2}\bm_2(1 + \OP(\sqrt{\varepsilon'})).
\end{eq}

\begin{proof}[Proof of Lemma~\ref{lem:app:b.4}]
Using~\eqref{eq:ders-gt} and the estimates derived in this section, we get that uniformly over all $t\in [0,1]$,
\begin{eq}\label{eq:app:gtder-arbitrary}
    |g_t'| &\lesssim  (n\mu)^{\frac{p}{r-1}} n^{\frac{1}{p}}\sqrt{\log n}\cdot\frac{\sigma^2}{\mu}\\
g_t'' 
&=\Big[p-1 +(r'-1)\Big(p(p-1) +\frac{1}{r-1} +1\Big)\Big] (n\mu)^{\frac{p}{r-1}-1}n^{-1+\frac{1}{p}}\bm_2(1+ \OP(\sqrt{\varepsilon'})).
\end{eq}
\end{proof}

\begin{proof}[Proof of Proposition~\ref{prop:direct-der}]
From~\eqref{eq:app:gtstder}, we can write
\begin{align*}
\frac{\dif}{\dif t} \eta_{n,t}( A_n)\bigg\vert_{t=0}
=\frac{d}{dt}\Big(\frac{g_t}{s_t}\Big) \bigg\vert_{t=0}
= n^{-1+\frac{1}{p} - \frac{1}{r}}\bm_1(1+o(1)).
\end{align*}
Also, using~\eqref{eq:app:etftstgt} and Lemmas~\ref{lem:app:b.3} and~\ref{lem:app:b.4}, we get
\begin{align*}
    \frac{\dif^2}{\dif t^2} \eta_{n,t}( A_n)
    =
    \frac{d^2}{dt^2}\Big(\frac{g_t}{s_t}\Big)=
    \Big[p-1 +\frac{1}{r-1}\Big]n^{-1+\frac{1}{p} - \frac{1}{r}}\frac{\bm_2}{n\mu}(1+ \OP(\sqrt{\varepsilon'})).
\end{align*}

\end{proof}}

\end{document}